\documentclass[12pt]{article}

\usepackage[british]{babel}
\usepackage[babel]{csquotes}
\usepackage{amsmath}
\usepackage{amssymb}
\usepackage{amsthm}
\usepackage{amsfonts}
\usepackage{hyperref}   

\usepackage{enumitem}
\setenumerate[1]{label=(\textup{\alph*})}
\setenumerate[2]{label=(\textup{\roman*})}

\newcommand{\overbar}[1]{\mkern 1.5mu\overline{\mkern-1.5mu#1\mkern-1.5mu}\mkern 1.5mu}

\newtheorem{defi}{Definition}[section]
\newtheorem{expl}[defi]{Example}
\newtheorem{thm}[defi]{Theorem}
\newtheorem{lem}[defi]{Lemma}
\newtheorem{prop}[defi]{Proposition}
\newtheorem{cor}[defi]{Corollary}
\newtheorem{rem}[defi]{Remark}

\DeclareMathAlphabet{\mathdutchcal}{U}{dutchcal}{m}{n}
\SetMathAlphabet{\mathdutchcal}{bold}{U}{dutchcal}{b}{n}
\DeclareMathAlphabet{\mathdutchbcal}{U}{dutchcal}{b}{n}

\DeclareMathOperator\diam{diam}
\DeclareMathOperator*{\argmax}{argmax}

\DeclareMathOperator*{\inj}{inj}

\newcommand{\R}{\mathbb{R}}         
\newcommand{\C}{\mathbb{C}}         
\newcommand{\N}{\mathbb{N}}         
\newcommand{\bb}{\mathcal{B}}       
\newcommand{\olb}{\overbar{\bb}}   
\newcommand{\mcl}{\mathcal{M}}      
\newcommand{\hyp}{\mathbb{H}}       
\newcommand{\sph}{\mathbb{S}}       

\newcommand{\dd}{~\textup{d}}       
\newcommand{\eps}{\varepsilon}
\newcommand{\vphi}{\varphi}
\newcommand{\ncl}{\mathcal{N}}      
\newcommand{\wcl}{\mathcal{W}}      
\newcommand{\eh}{\frac{1}{2}}
\newcommand{\cdu}{\mathdutchbcal{c}}
\newcommand{\can}{{\operatorname{can}}}
\newcommand{\ecl}{\mathcal{E}}
\newcommand{\hcl}{\mathcal{H}}
\newcommand{\kcl}{\mathcal{K}}

\newcommand{\ind}{1\hspace{-0.098cm}\mathrm{l}}     
\newcommand{\indi}[1]{\,\ind_{\{#1\}}}              

\newcommand\prob[1]{\mathbb{P}\left(#1\right)}	    
\newcommand\expec[1]{\mathbb{E}\left[#1\right]}	    
\newcommand\var[1]{\mathbb{V}\left[#1\right]}	    
\newcommand\eqd{{~\stackrel{(d)}{=}~}}     		    

\newcommand\abs[1]{\left\lvert #1 \right\rvert}     
\newcommand\abrac[1]{\left\langle #1 \right\rangle} 
\newcommand\norm[1]{\left\Vert #1 \right\Vert}      

\newcommand{\ol}[1]{\overbar{#1}}

\newcommand{\wh}[1]{\widehat{#1}}

\newcommand{\ocl}{\mathcal{O}}
\newcommand{\odcl}{\mathdutchcal{o}}

\title{Persistence probabilities of fractional L\'evy fields indexed by hyperbolic space and other Riemannian manifolds}
\author{Frank Aurzada and  Max Helmer}
\date{\today}

\begin{document}

\maketitle 

\begin{abstract}
    We study the persistence probability of fractional L\'evy fields, i.e.\ the analogue of fractional Brownian motion with generalised (multi-dimensional) index sets. First, we compute the persistence exponent of the hyperbolic fractional L\'evy field. The result matches the rate obtained in \cite{Molchan99} for Euclidean space and the one in \cite{AurzadaHelmer2026} for the sphere. This enables us to study persistence for fractional L\'evy fields indexed by a large class of Riemannian manifolds (whenever that process exists) through a local comparison argument with the spherical and hyperbolic case.
\end{abstract}


\textbf{Math subject classification (2020):} 60G22, 60G60 (primary); 60G15, 53B20 (secondary) 

\allowdisplaybreaks

\section{Introduction}

Persistence probabilities arise naturally in the broader study of stochastic processes under constraints: they describe the probability that a stochastic process or field remains below a prescribed level over an extended region of its index space. Persistence probabilities are highly sensitive both to the dependence structure of the field and to the geometry of the underlying index set. Explicit  results are typically difficult to obtain so that one focuses on the asymptotic (polynomial) decay rate which is governed by the persistence exponent. In this work, we study the interplay for fractional L\'evy fields indexed by hyperbolic space and more general Riemannian manifolds. The main question is whether the persistence exponent agrees with the exponent suggested by Euclidean fractional Brownian fields \cite{Molchan99}. This is a companion paper to \cite{AurzadaHelmer2026}, where the same question is studied for the sphere instead of the hyperbolic space.

We refer to the monograph \cite{MetzlerEtAl2014book} and the surveys \cite{BrayMajumdarSchehr2013, AurzadaSimon2015} for standard references on persistence probabilities. For (fractional) Gaussian fields we refer to \cite{LodhiaEtAl2016} for a survey. The main object of the present work is the fractional L\'evy field indexed by hyperbolic space, cf.\ \cite{Istas2005, Istas2012, Faraut1973}. This is a generalisation of the fractional Brownian motion defined by \cite{Kolmogoroff1940} and popularised by \cite{MandelbrotVanNess1968}. The naming convention \emph{L\'evy field} has been adopted from \cite{CohenLifshits2012}.

\paragraph*{\textit{Our framework.}} Let us first recall standard fractional Brownian motion on the real line.
Fractional Brownian motion is the centred, a.s.\ continuous Gaussian process $(B_H(t))_{t\in\R}$ with $B_H(0)=0$ a.s.\ and structure function
\begin{align*}
    \expec{(B_H(s) - B_H(t))^2} = \abs{t-s}^{2H}, \qquad s,t\in\R,
\end{align*}
where $H\in (0,1)$ is called the Hurst parameter. The case $H=1/2$ corresponds to Brownian motion.
Standard works on the topic are \cite{BiaginiEtAl2008book, Mishura2008book, Nourdin2012book}.

Substituting the index set for the multidimensional $\R^d$ and replacing $\abs{.}$ by $\norm{.}$, the Euclidean norm in $\R^d$, we obtain  fractional Brownian motion with multidimensional time, which is the centred, a.s.\ continuous Gaussian process also denoted by $(B_H(t))_{t\in\R^d}$, determined by $B_H(0)=0$ a.s.\ and by the structure function
\begin{align} \label{eqn:defeuclideanlf}
    \expec{(B_H(s) - B_H(t))^2} = \norm{t-s}^{2H}, \qquad s,t\in\R^d.
\end{align}
L\'evy was first to define this process \cite{Levy1940}, which is why the name \emph{fractional L\'evy field} or \emph{L\'evy Brownian field} (to denote $H=1/2$) has been adopted.

The focus in this work is on the persistence probability. Let us recall the main result for $(B_H(t))_{t\in \R^d}$ from \cite[Thm.\ 3]{Molchan99} (cf.\ introduction of \cite{Molchan2018} for the assumption on $K$): If $K$ is a bounded domain containing $0$ in its interior then 
\begin{align}\label{mol_classicalExpandingDomain}
    \prob{ \sup_{t\in TK} B_H(t) < 1} &= T^{-d+\odcl(1)}, &\text{as } T\to \infty,
\end{align}
where $T K := \{T s : s\in K\}$ is a linearly expanding domain. 

The fractional L\'evy field is self-similar with self-similarity parameter $H$, i.e.\ the processes $(B_H(c t))_{t\in\R^d}$ and $(\norm{c}^H B_H(t))_{t\in\R^d}$ possess the same finite dimensional distributions. Therefore, (\ref{mol_classicalExpandingDomain}) is equivalent to
\begin{align}\label{mol_dualFixedDomain}
    \prob{ \sup_{t\in K} B_H(t) < \eps } &= \eps^{\frac{d}{H} + \odcl(1)},  & \text{as } \eps\searrow 0.
\end{align}
The topic of this paper is to investigate more general index sets and to study the influence of the geometry of the index set on persistence. First, consider the sphere 
\begin{align*}
    \sph_{d} := \{\eta\in\R^{d+1} : \norm{\eta} = 1\}.
\end{align*}
Choose an arbitrary point $O\in\sph_d$. Then there exists a fractional L\'evy field for $0<H\leq 1/2$ with origin $O$, i.e.\ a centred, a.s.\ continuous Gaussian field $(S_H(\eta))_{\eta\in\sph_d}$ with $S_H(O)=0$ a.s.\ and  structure function
\begin{align}\label{eq_sphericalFractionalLevyFieldDef}
    \expec{(S_H(\eta) - S_H(\zeta))^2} &= d(\eta,\zeta)^{2H}, & \eta,\zeta\in\sph_d , 
\end{align}
where $d(.,.) := \arccos(\abrac{\eta,\zeta})$ is the shortest distance on the sphere and where $\abrac{.,.}$ denotes the usual Euclidean scalar product. Note that we use $d(.,.)$ in a generic way without a subindex when its meaning is evident from context. We generally write any metric $d(.,.)$ accompanied by parentheses to distinguish it from the fixed dimension $d$.

The existence of $(S_H(\eta))_{\eta\in\sph_d}$ in the Brownian case $H=1/2$ was first shown by L\'evy \cite{Levy1965book}. The existence in the other cases may be inferred from \cite[Sec.\ 5, Eq.\ (5.4)]{Gangolli1967}. A more modern and direct reference is \cite[Lem.\ 3.1]{Istas2005}. 
In contrast to the Euclidean case, such a field does not exist for $H>1/2$, which was shown in \cite[Thm.\ 2.1]{Istas2005}.
Since the sphere is compact, one studies the persistence exponent in the sense of \eqref{mol_dualFixedDomain} (rather than \eqref{mol_classicalExpandingDomain}). It was shown in \cite{AurzadaHelmer2026} that
\begin{align}\label{eq_sphericalPersistenceExponent}
    \prob{ \sup_{\eta\in \sph_d} S_H(\eta) < \eps } &= \eps^{\frac{d}{H} + \odcl(1)},  & \text{as } \eps\searrow 0,
\end{align}
so the geometry of the sphere does not affect the exponent.

The central object of this work is the fractional L\'evy field indexed by hyperbolic spaces. Define the \emph{hyperboloid}, a model of hyperbolic space, $\hyp_d$: If $\eta,\zeta\in\R^{d+1}$ and $\eta\circ \zeta := (\eta\circ\zeta):= -\eta_1 \zeta_1 + \eta_2 \zeta_2 + \ldots + \eta_{d+1} \zeta_{d+1}$ denotes the Lorentzian inner product (cf.\ \cite[Sec.\ §3.1]{Ratcliffe2019}) then the unit hyperboloid model $\hyp_d$ is given by
\begin{align}\label{eq_defHyperboloid}
    \hyp_d := \{ \eta\in\R^{d+1} : \eta_1 > 0, ~ \eta\circ \eta = -1\}
\end{align}
with the hyperbolic distance function given by $d(\eta,\zeta)= \operatorname{arccosh}(-\eta\circ \zeta)$, cf.\ \cite[Sec.\ §3.2]{Ratcliffe2019}.

Fix an arbitrary point $O\in \hyp_d$.  The fractional L\'evy field indexed by $\hyp_d$ is a centred, a.s.\ continuous Gaussian field $(X_H(\eta))_{\eta\in\hyp_d}$ with $X_H(O)=0$ a.s.\ and with structure function
\begin{align*} 
    \expec{(X_H(\eta) - X_H(\zeta))^2} &= d(\eta,\zeta)^{2H}, & \eta,\zeta\in\hyp_d, 
\end{align*}
in complete analogy to \eqref{eqn:defeuclideanlf} and \eqref{eq_sphericalFractionalLevyFieldDef}. 
Existence was proven in \cite{Faraut1973} in the case $H=1/2$. Using \cite[Lem.\ 2.1]{Istas2005} the existence of the corresponding fractional process follows for $0<H\leq 1/2$. The non-existence in the cases $H>1/2$ was shown in \cite[Thm.\ 4.1]{Istas2005}.

\paragraph*{\textit{Main results.}}
In our first main result we obtain the persistence exponent of the hyperbolic fractional L\'evy field on a bounded domain and find that the value is what we expect in analogy to \cite{Molchan99} and \cite{AurzadaHelmer2026}. 

\begin{thm}\label{thm_mainHyperbolic}
    For the hyperbolic fractional L\'evy field $(X_{H}(\eta))_{\eta\in\hyp_{d}}$ and a bounded set $K\subset \hyp_{d}$ with $O$ being any point in the interior of $K$ we have
    \begin{align} \label{eqn_mainHyperbolic}
        \prob{ \sup_{\eta\in K} X_{H}(\eta) < \eps}  =\eps^{\frac{d}{H} + \odcl(1)}, \qquad \text{ as } \eps\searrow 0.
    \end{align}
\end{thm}

The proof uses a refinement of the comparison technique established in \cite{AurzadaHelmer2026}, the classical result from \cite{Molchan99}, and a new existence result for the density of the argmax for a general stationary increment field on homogeneous Riemannian manifolds given as Theorem~\ref{thm_densityArgmaxRiemannian} below. The latter is a generalization of the results in \cite{SamorodnitskyShen2013} from the real line to a class of manifolds.

Our second main result generalises Theorem~\ref{thm_mainHyperbolic} from the hyperbolic index set to $d$-dimensional connected Riemannian manifolds $\mcl$ without boundary. If $O\in\mcl$ is an arbitrary point then let the fractional L\'evy field on $\mcl$ be defined as $(X_H(\eta))_{\eta\in\mcl}$ with $0<H<1$, with $X_H(O)=0$ a.s.\ and with structure function
\begin{align*}
    \expec{(X_H(\eta)-X_H(\zeta))^2} = d(\eta,\zeta)^{2H}, \qquad\qquad \text{for all } \eta,\zeta\in \mcl,
\end{align*}
where $d(.,.)$ is the (geodesic) shortest distance function on the manifold $\mcl$. We suppose that such a field exists for $\mcl$. Then, for the range of Hurst parameters $0<H\leq  1/2$, we are able to obtain the persistence exponent.

\begin{thm}\label{thm_mainRiemannianManifold}
    For $0<H\leq 1/2$ let $(X_{H}(\eta))_{\eta\in\mcl}$ be the fractional L\'evy field on the $d$-dimensional connected Riemannian manifold $\mcl$ without boundary. For any relatively compact set $K\subset \mcl$ with $O$ in the interior of $K$ we have
    \begin{align*} 
        \prob{ \sup_{\eta\in K} X_{H}(\eta) < \eps}  =\eps^{\frac{d}{H} + \odcl(1)}, \qquad \text{ as } \eps\searrow 0.
    \end{align*}
\end{thm}

The previous theorem hinges on the question of the existence of a fractional L\'evy field on a given manifold. This question is non-trivial in general. The next lemma gives an existence result for such fractional L\'evy fields \emph{locally} on a given \emph{$2$-dimensional} manifold under some additional regularity assumption. 

\begin{lem}\label{lem_mainExistence2dFBM}
    For any connected, $2$-dimensional Riemannian manifold without boundary and an arbitrary point $O$ on it, there exists a neighbourhood $U$ around $O$ with a fractional L\'evy field on $U$ for $0<H\leq 1/2$. 
\end{lem}

The result is based on ideas discussed in \cite{ChentsovMorozova1968} and \cite[Sec.\ 2]{VenetThesis2016}, where this is elaborated for a related class of manifolds. We provide a proof for Lemma~\ref{lem_mainExistence2dFBM} that is self-contained, except for one deeper geometric result, in the appendix.

\paragraph*{\textit{Outline of this article.}} We first give a brief overview of some related work and open problems in Section~\ref{sec_relatedWorkOpenProblems}. In Section~\ref{sec_argmaxDensity}, we lift a result on the existence of Lebesgue densities for the argmax of general stationary increment fields from \cite{SamorodnitskyShen2013} to a Euclidean multidimensional index setting. Then, in Section~\ref{sec_argmaxDensityRiemann}, we provide the analogous statement for stationary increment fields indexed by a class of Riemannian manifolds after a brief background on necessary concepts from Riemannian geometry. In Section~\ref{sec_persistenceHyperbolic}, we apply the results of the previous section to obtain the persistence exponent of the hyperbolic fractional L\'evy field and thus prove Theorem~\ref{thm_mainHyperbolic}. This allows us to infer the value of the persistence exponent of general fractional L\'evy fields indexed by Riemannian manifolds in Section~\ref{sec_persistenceGeneralRiemann}, which proves Theorem~\ref{thm_mainRiemannianManifold}. 

\subsection{Related work and open problems}\label{sec_relatedWorkOpenProblems}

\paragraph*{\textit{Further problems in the Euclidean case.}} The fundamental result \eqref{mol_classicalExpandingDomain} is accompanied by further ones and some longer standing open questions. In \cite[Thm.\ 1]{Molchan99} it was shown that
\begin{align}\label{mol_oneSided}
    \prob{ \sup_{t\in [0,T]} B_H(t) < 1} &= T^{-(1-H)+\odcl(1)}, &\text{as } T\to \infty.
\end{align}
Note that by simply having $0$ at the edge of the interval, instead of the interior, the Hurst parameter enters the persistence exponent. This is accompanied by the result in \cite{Molchan2018}, where it is shown that if $K$ is a $d$-dimensional bounded convex domain and if $K$ has a smooth boundary at zero, i.e.\ $0\in\partial K$ and there is a ball $\ol{\bb_r(t_0)}\subseteq K$ with some $t_0\in K$ then
\begin{align}\label{mol_smoothConvexBoundary}
    \prob{ \sup_{t\in T K} B_H(t) < 1} &= T^{-(d-H)+\odcl(1)}, &\text{as } T\to \infty.
\end{align}
If we write the exponent as $(d-1) + (1-H)$ and compare to the exponents in \eqref{mol_classicalExpandingDomain} and \eqref{mol_oneSided}, one may ask whether for domains like $K = [0,1]^k \times [-1,1]^{d-k}$ it is true that
\begin{align}\label{eq_molchanCojecture}
    \prob{ \sup_{t\in T K} B_H(t) < 1} &= T^{-(d - kH)+\odcl(1)}, &\text{as } T\to \infty.
\end{align}
This was conjectured by Molchan in \cite{Molchan2017} and is still an open problem.

\paragraph{\textit{The classical persistence problem in the hyperbolic case.}} Since hyperbolic space is non-compact, it makes sense to ask the classical question of persistence with respect to  an expanding domain, i.e.\ similar to the original result \eqref{mol_classicalExpandingDomain} by Molchan. Unfortunately, our technique cannot be applied in this case. For more details, cf.\ Remark~\ref{rem_hyperbolicDualProblem} below.

\paragraph{\textit{Existence and non-existence statements.}} After \cite{Levy1940} (cf.\ \cite{Levy1965book}) the question of existence of fractional L\'evy fields on Riemannian manifolds has been considered from a few different perspectives. For a spectral approach, cf.\ \cite{Gangolli1967, Faraut1973, Faraut1974, AskeyBingham1976} and \cite{Molchan1967, Molchan1979, CohenLifshits2012}. Geometric constructions can be found in \cite{Chentsov1957, ChentsovMorozova1968, TakenakaKuboUrakawa1981, Takenaka1987, Takenaka1991, Istas2006stable}. Other notable contributions also with a focus on non-existence statements are \cite{Istas2005, Istas2012, VenetThesis2016}. A related approach showing that self-similarity and stationary increments of a Gaussian field no longer uniquely characterize fractional L\'evy fields on Riemannian manifolds is \cite{Gelbaum2014}. There is also an interest in this topic from a purely non-stochastic perspective, cf.\ \cite{HjorthKokkendorffMarkvorsen2002}.

Therefore it is natural to ask the following question: Can (fractional) L\'evy fields be defined locally on a large class of higher dimensional Riemannian manifolds? To our knowledge, except for \cite{ChentsovMorozova1968}, no local existence questions have been considered and counterarguments to existence, except for \cite{Istas2012, FeragenLauzeHauberg2015} with $H=1$, use global properties. 

\paragraph*{\textit{The case $1/2<H<1$.}} It would be interesting to have an example for the existence of a fractional L\'evy field for $1/2<H<1$ on a Riemannian manifold that is not at least locally \emph{isometric} to Euclidean space. Such examples exist on trees, cf.\ \cite[Thm.\ 2.15]{Istas2012}. Only for $H=1$ we know that there exists none, cf.\ \cite[Thm.\ 1]{FeragenLauzeHauberg2015}.

\section{Lebesgue density of the argmax of real-valued stationary increment processes indexed by Euclidean space}\label{sec_argmaxDensity}

This section is a special case of Section~\ref{sec_argmaxDensityRiemann}. The arguments are essentially the same, but the technical details are more approachable in the Euclidean setting.

In this section, let $(X_t)_{t\in\R^d}$ be a real-valued separable stochastic process that has stationary increments 
and that attains its maximum on any compact set a.s.\ at a unique point. Then for any relatively compact set $A\subseteq \R^d$ the random variable
\begin{align*}
    \tau_{\ol{A}} := \argmax_{s \in \ol{A}} X_s,
\end{align*}
is a.s.\ well-defined. We write $\ol{A}$ for the closure of $A$. 

The goal of this section is to show that for open and bounded sets $A$, which are therefore relatively compact, the random variable $\tau_{\ol{A}}$ is absolutely continuous with respect to  the Lebesgue measure on $A$. The assumption of an a.s.\ unique maximum is unrestrictive in our context (cf.\ Corollary~\ref{cor_densityArgmaxHyperbolic} below) due to the statement from \cite[Lem.\ 2.6]{KimPollard1990}, which we quote as Lemma~\ref{lem_uniqueMaximum} below. Generalisations may be found in \cite{Arcones1992} and \cite{LopezPimentel2018}. Note that a $\sigma$\textit{-compact} metric space means that it is the countable union of compact sets. Euclidean space is obviously $\sigma$-compact and every other manifold discussed in this article will possess this property, as well.
\begin{lem}\label{lem_uniqueMaximum}
    Let $(E,d)$ be a $\sigma$-compact metric space and let $(X_t)_{t\in E}$ be a Gaussian process with a.s.\ continuous sample paths. If $\var{X_s - X_t} \neq 0$ for all $s\neq t$ in $E$ then the supremum of the process can a.s.\ never be attained at two different points.
\end{lem}

We use $\bb_r(t) := \{ s\in\R^d : \norm{t-s} < r\}$
to denote open balls with a centre $t$ of a radius $r$. 
The following geometric auxiliary lemma is key to the main theorem of this section. We will not give a proof here, but instead refer to the analogous statement in a generalised setting, cf.\ Lemma~\ref{lem_epsBallNumber_Riemann}.

\begin{lem}\label{lem_epsBallNumber}
    Let $r>0$ and let $0 < \eps < r$. There is a constant $\mathdutchbcal{c} > 0$ depending only on $d$, such that the number of disjoint Euclidean balls in $\R^d$ of radius $\eps$ that can fit into $\bb_r(0)$ is at least $\mathdutchbcal{c} ~ r^d ~ \eps^{-d}$.
\end{lem}

We also require the following standard argument about the existence of Lebesgue densities. The (almost identical) proof of the more general statement made in Lemma~\ref{lem_ballInequalityImpliesSurfaceDensity} below may be found in the appendix.

\begin{lem}\label{lem_ballInequalityImpliesLebesgue}
    Let $X$ be a random variable in $\R^d$. Let $A\subseteq \R^d$ be an open set and suppose there exist constants $c_A>0$ and $\eps_A > 0$ (that may both depend on $A$), such that for all $0<\eps<\eps_A$ and all $x\in A$ we have
    \begin{align*}
        \prob{X\in \bb_\eps (x)} \leq c_A ~ \eps^d.
    \end{align*}
    Then the distribution of $X$ is absolutely continuous with respect to  the Lebesgue measure on $A$.
\end{lem}

We can now prove the following theorem, which essentially follows the proof of Theorem 3.1 in \cite[pp.\ 3499]{SamorodnitskyShen2013} for stationary processes on the real line in a more generalised setting. In the proof of \cite[Lem.\ 2.8]{Shen2016} it was remarked that it analogously works for stationary increment processes. For different possible assumptions, cf.\ Remark~\ref{rem_generalisationArgmax} below. 

\begin{thm}\label{thm_densityArgmaxEuclidean}
    Let $(X_t)_{t\in\R^d}$ be a real-valued, separable stochastic process that has stationary increments and that attains its maximum on any compact set a.s.\ at a unique point. Then for any $T>0$ the random variable $\tau_{\olb_T(0)}:= \argmax_{s \in \olb_T(0)} X_s$ has a Lebesgue density in $\bb_T(0)$.
\end{thm}
\begin{proof}
    Choose $0 < \delta < T/2$. We are going to show that for any $t\in\bb_{T-\delta}(0)$ and any $0< \eps < \delta/2$ the inequality
    \begin{align}\label{eq_argmaxDensityUpperBound}
        \prob{\tau_{\olb_T(0)}\in \bb_\eps(t)} \leq \mathdutchbcal{c}^{-1} ~ \eps^d ~ 2^d ~ \delta^{-d}
    \end{align}
    holds. Here, $\mathdutchbcal{c}$ is the constant given in Lemma~\ref{lem_epsBallNumber}.
    Then we will be able to apply Lemma~\ref{lem_ballInequalityImpliesLebesgue} with $A:=\bb_{T-\delta}(0)$, with $c_A:=\mathdutchbcal{c}^{-1} ~ 2^d ~ \delta^{-d}$ and with $\eps_A:=\delta/2$ to conclude that the argmax has a Lebesgue density on $\bb_{T-\delta}(0)$. Since we can choose $\delta$ arbitrarily small, we will obtain the existence of a Lebesgue density on the whole $\bb_T(0)$.

    We are going to proceed by contradiction. Suppose that there is a $\delta < T/2$, a $t\in\bb_{T-\delta}(0)$ and a $0< \eps < \delta/2$, such that \eqref{eq_argmaxDensityUpperBound} fails. According to Lemma~\ref{lem_epsBallNumber}, we can find points $s_i$ with $i=1, \ldots, \lceil \cdu (\delta/2)^d \eps^{-d} \rceil$, such that all balls $\bb_\eps(s_i)$ are disjoint subsets of $\bb_{\delta/2}(t)$. At the same time, we have for all $i,j\in\{1, \ldots, \lceil \cdu (\delta/2)^d \eps^{-d} \rceil\}$ that
    \begin{align}\label{eq_ballInclusion}
        \bb_\eps(s_i) \subseteq \bb_T(s_j - t),
    \end{align}
    since for any $w\in\bb_\eps(s_i)$ we can calculate that
    \begin{align*}
        \norm{w-(s_j-t)}\leq \norm{w-s_j} + \norm{t} \leq \norm{w-t}+\norm{t-s_j}+\norm{t}< \frac{\delta}{2} + \frac{\delta}{2} + (T-\delta) = T.
    \end{align*}
    To obtain information about the distribution of the argmax we use the separability and the stationary increments of the process to obtain for any $u\in\R^d$ that
    \begin{align}\label{eq_EuclideanArgmaxShiftStaionaryIncr}
        \tau_{\olb_T(0)} = \argmax_{s\in \olb_T(0)} (X_s) = \argmax_{s\in \olb_T(0)} (X_s - X_u) \eqd \argmax_{s\in \olb_T(0)} (X_{s-u}) = \argmax_{s\in \olb_T(-u)} (X_{s}) + u = \tau_{\olb_T(-u)} + u.
    \end{align}
    In particular, for $u:= t-s_i$ we may infer that
    \begin{align*}
        \prob{\tau_{\olb_T(s_i - t)}\in \bb_\eps(s_i)} = \prob{\tau_{\olb_T(0)}- (t-s_i)\in \bb_\eps(s_i)} = \prob{\tau_{\olb_T(0)}\in \bb_\eps(t)}
    \end{align*}
    for every $i=1, \ldots, \lceil \cdu (\delta/2)^d \eps^{-d} \rceil$. Furthermore, for $i\neq j$, we will check that these events may only intersect on a set of measure zero. For this purpose, define the events
    \begin{align*}
        \Omega_{s_i,s_j} & := \{ \tau_{\olb_T(s_i - t)}\in \bb_\eps(s_i), ~\tau_{\olb_T(s_j - t)}\in \bb_\eps(s_j) \}. 
    \end{align*}
    We show that this is a set of measure zero for all $i\neq j$. The set
    \begin{align}\label{eq_OmegaSiSjEuklidisch}
        \Omega_{s_i,s_j} \cap \{ X_{\tau_{\olb_T(s_i-t)}} < X_{\tau_{\olb_T(s_j-t)}} \}
    \end{align}
    must be empty, since by the definition of $\Omega_{s_i,s_j}$ and by \eqref{eq_ballInclusion} we would have
    \begin{align*}
        \tau_{\olb_T(s_j-t)} \in \bb_\eps(s_j) \subseteq \olb_T(s_i-t),
    \end{align*}
    which would imply that $X_{\tau_{\olb_T(s_i-t)}} \geq X_{\tau_{\olb_T(s_j-t)}}$, contradicting the condition in the second set in \eqref{eq_OmegaSiSjEuklidisch}. For the same reason 
    \begin{align*}
        \Omega_{s_i,s_j} & \cap \{ X_{\tau_{\olb_T(s_i-t)}} > X_{\tau_{\olb_T(s_j-t)}} \}
    \end{align*}
    must be the empty set, as well. Using the a.s.\ uniqueness of the maximum,
    \begin{align*}
        \Omega_{s_i,s_j} & \cap \{ X_{\tau_{\olb_T(s_i-t)}} = X_{\tau_{\olb_T(s_j-t)}} \}
    \end{align*}
    must be a set of measure zero. Thus, we conclude that $\Omega_{s_i,s_j}$ is a set of measure zero for all $i\neq j$.

    Finally, we apply our knowledge about the sets $\Omega_{s_i,s_j}$, our assumption that \eqref{eq_argmaxDensityUpperBound} fails and the assumptions about our constants to obtain
    \begin{align*}
        1 & \geq \prob{\bigcup_{i=1}^{\lceil \cdu (\delta/2)^d \eps^{-d} \rceil} \{\tau_{\olb_T(s_i-t)} \in \bb_\eps(s_i) \} } \\
        & = \sum_{i=1}^{\lceil \cdu (\delta/2)^d \eps^{-d} \rceil} \prob{ \tau_{\olb_T(s_i-t)} \in \bb_\eps(s_i) } \\
        & = \sum_{i=1}^{\lceil \cdu (\delta/2)^d \eps^{-d} \rceil} \prob{ \tau_{\olb_T(0)} \in \bb_\eps(t) } \\
        & > \cdu ~ (\delta/2)^d ~ \eps^{-d} ~ \cdu^{-1} ~ \eps^d ~ 2^d ~ \delta^{-d} \\
        & =  1,
    \end{align*}
    which is a contradiction. Thus, \eqref{eq_argmaxDensityUpperBound} holds and the theorem is proved.
\end{proof}

\begin{rem}
    Note that the Lebesgue density in Theorem~\ref{thm_densityArgmaxEuclidean} might be zero on all of $\bb_T(0)$ and that the mass of $\argmax$ might thus be concentrated on the boundary of $\olb_T(0)$. An easy example for this is the Gaussian process $(t~ Y)_{t\in\R}$, where $Y \sim \ncl(0,1)$ (i.e.\ fractional Brownian motion on the real line with $H=1$). Then $\argmax_{s\in [-T,T]} s Y = \indi{Y>0} T + \indi{Y<0} (-T)$ almost surely.
\end{rem}

\begin{rem}\label{rem_generalisationArgmax}
    (a) Theorem~\ref{thm_densityArgmaxEuclidean} also holds for separable \emph{stationary} processes $(X_t)_{t\in\R^d}$ with an a.s.\ unique maximum. The only part where we required the stationary increments is \eqref{eq_EuclideanArgmaxShiftStaionaryIncr} and this can be readily proved under the assumption of stationarity. 

    (b) If instead of an a.s.\ unique maximum we demand upper semi-continuous sample paths then the $\argmax$ is not defined as above, but one may introduce a strict order on the components of $\R^d$ and thus still get a well-defined $\tau_{\ol{A}}$.
    Thus, the same argument will show that $\tau_{\olb_T(O)}$ possesses a Lebesgue density in $\bb_T(O)$ for any $T>0$ (cf.\ \cite[pp. 3499]{SamorodnitskyShen2013}).
\end{rem}

\section{Existence of the argmax-density on homogeneous Riemannian manifolds}\label{sec_argmaxDensityRiemann}

We are now going to apply the same proof technique as in the last section in a further generalised setting. One key assumption is that the process is indexed by a homogeneous Riemannian manifold and we are going to explain what the notion of stationary of increments means in this context. We start with a brief overview of the necessary background from Riemannian geometry.

\subsection{Brief background from differential and Riemannian geometry}\label{sec_RiemannianBackground}

As our main references for Riemannian geometry we are going to use \cite{Chavel2006} and \cite{Petersen2016}. For an introduction to the topic of manifolds and differential geometry in general, cf.\ \cite{Lee2012book}. We provide a quick overview of concepts from differential geometry following the definitions in \cite[Sec.\ 1-3]{Lee2012book}.

A \emph{($d$-dimensional) topological manifold} $\mcl$ is a Hausdorff space with a second countable base so that for every $\eta\in\mcl$ there is an open subset $U\subseteq\mcl$ with $\eta\in U$, an image set $\wh{U}\subseteq\R^d$  and a homeomorphism $\vphi:U\to \wh{U}$. Such a pair $(U,\vphi)$ (with $\wh{U}:=\vphi(U)$) is called a chart. A map between Euclidean spaces is called \emph{smooth} if it has continuous partial derivatives of all orders. Two charts $(U,\vphi)$, $(V, \psi)$ are called \emph{smoothly compatible}, if $U\cap V=\emptyset$ or if $\psi\circ\vphi^{-1}:\vphi(U\cap V)\to\psi(U\cap V)$ is a smooth map. A \emph{maximal smooth atlas} for a topological manifold $\mcl$ is a collection of charts $(U_a,\vphi_a)_{a\in I}$ so that the domains $(U_a)_{a\in I}$ cover $\mcl$, so that all charts are smoothly compatible, and so that all charts smoothly compatible with every chart in the atlas are already part of the atlas. A topological manifold $\mcl$ together with a maximal smooth atlas is called a \emph{smooth / differentiable / $\mathcal{C}^\infty$} manifold. This definition implies that there is no boundary $\partial \mcl$.

A map $f:\mcl\to\R^k$ is \emph{smooth}, if for every $\eta\in\mcl$ there is a chart $(U,\vphi)$ of $\mcl$ with $\eta\in U$ so that $f\circ\vphi^{-1}:\vphi(U)\to\R^k$ is smooth. A map $F:\mcl\to\ncl$ between smooth manifolds is \emph{smooth} if for every $\eta\in\mcl$ there is a chart $(U,\vphi)$ with $\eta\in U$ and a chart $(V,\psi)$ with $F(\eta)\in V$ so that $F(U)\subseteq V$ and so that $\psi\circ F\circ \vphi^{-1}:\vphi(U)\to\psi(V)$ is smooth. A smooth function from a real interval $[a,b]$ into a smooth manifold, where we only consider one-sided derivatives at the edges of the interval, is called a \emph{smooth curve}. 

The class of real-valued smooth functions is denoted by $\mathcal{C}^\infty(\mcl)$. A linear map $v:\mathcal{C}^\infty(\mcl)\to\R$ is called a \emph{derivation at $\eta$} if it satisfies $v(f g) = f(\eta) v(g)+g(\eta)v(f)$ for all $f,g\in\mathcal{C}^\infty(\mcl)$. The set of all derivations at $\eta$ is denoted by the tangent space $T_\eta\mcl$ and may be identified as a vector space with $\R^d$. One may see tangent spaces are the generalisation of directional derivatives to manifolds.

Following \cite{Chavel2006}, all manifolds $\mcl$ considered will be $d$-dimensional, smooth, connected and without boundary. Unless explicitly stated otherwise, all arguments will be applicable in any dimension.

Let $\mcl$ be a differentiable manifold, such that for each $\eta\in\mcl$ there exists a real-valued scalar product $\abrac{.,.}_\eta$ on the tangent space $T_\eta \mcl$. If these scalar products depend smoothly on $\eta\in\mcl$ this collection of scalar products, written as $\abrac{.,.}_\mcl$, is called a \textit{Riemannian metric}, cf.\ \cite[§I.5]{Chavel2006}. It induces the norm $\norm{t}_\mcl := \sqrt{\abrac{t,t}_\mcl}$ for $t\in T_\eta\mcl$ for any $\eta\in\mcl$. The tuple $(\mcl,\abrac{.,.}_\mcl)$ is called a \emph{Riemannian manifold}. We will often omit the Riemannian metric and just call $\mcl$ the Riemannian manifold.

The Riemannian metric allows one to make sense of length of curves on manifolds. Let $\gamma:[a,b]\to\mcl$ be a piecewise smooth curve, which is to say that there exists a $k\in\N$ and $a=a_0<a_1<\ldots<a_k=b$ so that for for each $i=1,\ldots,k$ the restriction $\gamma|_{[a_{i-1}, a_i]}$ defines a smooth curve. Then to any $t_0\in[a,b]$, where $\gamma$ is smooth, there is a \emph{velocity of $\gamma$ at $t_0$} in $T_{\gamma(t_0)}\mcl$, i.e.\ a generalised directional derivative, associated with it and denoted by $\gamma'(t_0)$. We use this to define
\begin{align}\label{eq_curveLengthDef}
    L(\gamma) := \int_a^b \norm{\gamma'(t)}_\mcl \dd t
\end{align}
as the \textit{length} of $\gamma$. 

A Riemannian manifold possesses a natural \emph{connection} between its tangent spaces determined by the Riemannian metric. This allows one to follow the velocity vectors along a curve through different tangent spaces. In this sense one defines a \emph{geodesic} to be a smooth curve with constant velocity vector. In particular, if $\gamma$ is a geodesic then there is a $c>0$ so that $\norm{\gamma'(t)}_\mcl \equiv c$ for all times $t$ where it is defined. Through standard ODE theory (cf.\ \cite[Sec.\ §I.3]{Chavel2006}) one obtains that a geodesic is always (at least locally) uniquely determined by a point $\eta\in\mcl$ together with a direction $t\in T_\eta\mcl$. In Euclidean space, a geodesic is just a straight line. On the sphere, geodesics follow the trace of the great circles (cf.\ Example~\ref{expl_shortestDistancesAndIsometriesModelSpaces}).

Let $\mathcal{C}(\eta,\zeta)$ be the set of piecewise smooth curves $\gamma$ defined on some non-empty interval $[a,b]$ with $\gamma(a)=\eta$ and $\gamma(b)=\zeta$. The \textit{geodesic distance} or \textit{minimal distance} for points $\eta,\zeta\in\mcl$ is now defined as
\begin{align*}
    d(\eta,\zeta) := \inf_{\gamma\in\mathcal{C}(\eta,\zeta)} L(\gamma)
\end{align*}
The space $\mcl$ is a metric space with the geodesic metric $d(.,.)$. The (open) \textit{geodesic distance balls} $\bb_r(\eta)$ for $r>0$ and $\eta\in\mcl$ are defined by
\begin{align*}
    \bb_r(\eta):= \{\xi\in\mcl : d(\eta,\xi) < r\}.
\end{align*}
When the space is ambiguous we may write $\bb_r^\mcl(\eta)$ to make it precise.

A fundamental concept in Riemannian geometry is the exponential map (cf.\ \cite[Sec.\ 5.5.1]{Petersen2016} or \cite[Thm.\ I.3.2]{Chavel2006}). For any $\eta\in\mcl$ this is a function
\begin{align}\label{eq_expMapDef}
    \exp_\eta: T_\eta \mcl \to\mcl
\end{align}
that maps the tangent space in a natural way to the manifold: If $t\in T_\eta\mcl$ then there is a geodesic $\gamma_t$ starting in $\eta$ and with constant velocity vector $t$. The value of the exponential map $\exp_\eta$ at $t$ is then $\gamma_t(1)$. In other words, $\exp_\eta(t)$ is the point one arrives at if one follows the direction $t$ starting in $\eta$ in a straight line for a unit time interval.

Key is the \textit{injectivity radius} $\inj(\eta)$ (or $\inj_\mcl(\eta)$ to avoid ambiguity), cf.\ \cite[p. 118]{Chavel2006} or \cite[Sec.\ 5.7.4]{Petersen2016}, which is the supremum of all radii $r>0$ such that $\exp_\eta |_{\bb_r^{T_\eta\mcl}(\eta)}$ is an injective function. For $r < \inj(\eta)$ the image is $\bb_r^\mcl(\eta)$ and, in particular, for any $t\in \bb_r^{T_\eta\mcl}(0)$ we have
\begin{align}\label{eq_expMapDistancePreserving}
    \norm{t}_\mcl = d(\eta, \exp_\eta(t)).
\end{align}
This identity will be applied often throughout this work. The injectivity radius $\inj(\eta)$ is continuous as a function on $\mcl$ and non-zero at any point $\eta\in\mcl$ (cf.\ \cite[Thm.\ III.2.3]{Chavel2006}). For sets $K\subseteq \mcl$ we define $\inj(K):=\inf_{\eta\in K} \inj(\eta)$. 

We are going to need to refer to a certain class of manifolds, which we obtain in a natural way from $\sph_d$ and $\hyp_d$: The sphere of radius $r$ is defined as
\begin{align}\label{eq_defSphereWithRadiusAndDistance}
    \sph_d(r) &:= \{ \eta \in \R^{d+1} : \norm{\eta}^2 = r^2\},
\end{align}
with $d(\eta,\zeta) := r \arccos\left(\frac{1}{r^2} \abrac{\eta,\zeta}\right)$ for $\eta,\zeta\in\sph_d(r)$.

The (unit) hyperboloid was defined in \eqref{eq_defHyperboloid}. Analogously we may define the hyperboloid of (imaginary) radius $r$ as
\begin{align}\label{eq_defHyperboloidWithRadiusAndDistance}
    \hyp_d(r) &:= \{ \eta\in\R^{d+1} : \eta_1 > 0, ~ \eta\circ \eta = -r^2\},
\end{align}
with $d(\eta,\zeta):= r \operatorname{arccosh}\left(- \frac{1}{r^2} ~ \eta\circ\zeta \right)$ for $\eta,\zeta\in\sph_d(r)$.

We can now define the \textit{model spaces} $\mcl_\kappa$ for any $\kappa\in\R$ as
\begin{align}\label{eq_defModelSpace}
    \mcl_\kappa :=
    \begin{cases}
        \sph_d(1/\sqrt{\kappa}) & \text{ for $\kappa > 0$},\\
        \R^d & \text{ for $\kappa = 0$}, \\
        \hyp_d(1/\sqrt{-\kappa}) & \text{ for $\kappa < 0$}.
    \end{cases}
\end{align}

The fact that we chose $1/\sqrt{\kappa}$ for the radii stems from the notion of sectional curvature in Riemannian geometry (cf.\ \cite[§II.1]{Chavel2006}). For more on the derivation of the distance function on the model spaces, cf.\ \cite[§II.3]{Chavel2006}. The injectivity radii on any of these spaces are identical at every point of the respective space. In particular, 
\begin{align}\label{eq_modelSpace_injRadius}
    \inj(\mcl_\kappa) =
    \begin{cases}
        \frac{\pi}{\sqrt{\kappa}} & \text{ for $\kappa > 0$},\\
        \infty & \text{ for $\kappa \leq 0$}.
    \end{cases}
\end{align}

On a Riemannian manifold there exists a \textit{Riemannian measure}, which we will call $\sigma$. Roughly speaking, it measures the Riemannian volume by mapping it into Euclidean space and taking the Lebesgue-measure with respect to  a homeomorphic transformation whose Jacobian is defined through the Riemannian metric $\abrac{.,.}_\mcl$. This is independent of the chart taken. For a detailed and elaborate construction and many properties, cf.\ \cite[Sec.\ XII]{AmannEscher2009book}. Alternatively, \cite[Sec.\ §III.3]{Chavel2006} and \cite[Thm.\ 2]{Christensen1970} provide essential statements.

We require the following statement about ball volumes with respect to  the Riemannian measure.

\begin{lem}\label{lem_ballVolumeBoundsGeneralRiemann}
    Let $\mcl$ be a Riemannian manifold and let $A\subseteq\mcl$ be relatively compact. Then there is $0<r\leq\inj(A)$ and constants $c,C>0$ depending on $A$, such that 
    \begin{align*}
        c~\eps^d \leq \sigma(\bb_\eps^{\mcl}(\eta)) \leq C ~ \eps^{d}.
    \end{align*}
    holds uniformly for all $\eta\in A$ and all $0<\eps<r$. If $\mcl = \mcl_\kappa$ with $\kappa\leq 0$ then any finite $r>0$ is admissible.
\end{lem}

In Lemma~\ref{lem_ballVolume} in the appendix we have given this statement in the case that $\mcl$ is one of the model spaces. The general statement follows from the well-known Riemannian volume comparison theorems \cite[Thm.\ III.4.2]{Chavel2006} and \cite[Thm.\ III.4.4]{Chavel2006} together with the bounds on the Riemannian volumes for the model spaces.

This lets us generalise Lemma~\ref{lem_epsBallNumber}.
\begin{lem}\label{lem_epsBallNumber_Riemann}
    For $r>0$ given by Lemma~\ref{lem_ballVolumeBoundsGeneralRiemann} there exists a constant $\mathdutchbcal{c} > 0$, such that the number of disjoint balls of radius $0 < \eps \leq r$ that can fit into $\bb_r(O)$ is at least $\mathdutchbcal{c} ~ r^d ~ \eps^{-d}$.
\end{lem}
\begin{proof}
    Suppose that $N$ is the maximal number of points $\eta_1, \ldots, \eta_N$, such that the balls $(\bb_\eps(\eta_i))_{i=1,\ldots,N}$ of radius $0<\eps<r$ are disjoint and such that
    \begin{align*}
        \bigcup_{i=1}^N \bb_\eps(\eta_i) \subseteq \bb_r(O).
    \end{align*}
    For $r/2 \leq \eps \leq r$ it suffices to choose $\mathdutchbcal{c} \leq 2^{-d}$ to satisfy the trivial bound $N\geq 1 \geq \mathdutchbcal{c} ~ r^d ~ \eps^{-d}$. Thus, let $0<\eps<r/2$. Then if we double the $\eps$-radii we obtain a covering of $\bb_{r-\eps}(O)$, i.e.\
    \begin{align}\label{eq_packingOfEpsBalls}
        \bb_{r-\eps}(O) \subseteq \bigcup_{i=1}^N \bb_{2\eps}(\eta_i).
    \end{align}
    Indeed, if that was not the case, we would have a $\zeta\in\bb_{r-\eps}(O)$ with $d(\zeta,\eta_i) > 2\eps$ for all $i=1,\ldots, N$. Since the distance between $\zeta$ and the boundary $\partial \bb_{r}(O)$ is also at least $\eps$, we could thus fit the ball $\bb_\eps(\zeta)$ disjointly with the other $N$ balls into $\bb_r(O)$. This would contradict the maximality of $N$. Thus, \eqref{eq_packingOfEpsBalls} holds and yields that
    \begin{align*}
        \sigma(\bb_{r-\eps}(O)) \leq \sum_{i=1}^N \sigma(\bb_{2\eps}(\eta_i)) \leq N ~ \sup_{i=1,\ldots,N} \sigma(\bb_{2\eps}(\eta_i)).
    \end{align*}
    Lemma~\ref{lem_ballVolumeBoundsGeneralRiemann} lets us obtain that there are constants $c,C>0$, such that
    \begin{align*}
        c (r-\eps)^d \leq N ~ C~ (2\eps)^d.
    \end{align*}
    Therefore, if we choose $\cdu:= \frac{c}{C}4^{-d}$ then we obtain that $\cdu r^d \eps^{-d}\leq 1 \leq N$ for all $r/2\leq \eps\leq r$ and that $\cdu r^d \eps^{-d}\leq \frac{c (r-\eps)^d}{C(2\eps)^d} \leq N$ for all $0<\eps<r/2$, which proves the statement.
\end{proof}

We may also exchange the Lebesgue measure for the Riemannian measure in Lemma~\ref{lem_ballInequalityImpliesLebesgue}. The proof can be found in the appendix.

\begin{lem}\label{lem_ballInequalityImpliesSurfaceDensity}
    Let $X$ be a random variable with values in $\mcl$ with Riemannian measure $\sigma$. Let $A\subseteq \mcl$ be an open set and suppose there exist constants $c_A>0$ and $\eps_A > 0$ (that may both depend on $A$), such that for all $0<\eps<\eps_A$ and all $x\in A$ we have
    \begin{align*}
        \prob{X\in \olb_\eps (x)} \leq c_A ~ \eps^d.
    \end{align*}
    Then the distribution of $X$ is absolutely continuous on $A$ with respect to  the Riemannian measure $\sigma$.
\end{lem}

\subsection{Homogeneous Riemannian manifolds}\label{sec_homogeneousRiemannianManifolds}

From now on we assume unless otherwise stated that $\mcl$ is a $d$-dimensional homogeneous Riemannian manifold with an arbitrary reference point $O\in\mcl$. We recall the definition of a homogeneous manifold (cf.\ \cite[p. 59]{Chavel2006}).

\begin{defi}\label{def_homogeneousRiemannianManifold}
    A Riemannian manifold $\mcl$ is called homogeneous if there is a group $\mathcal{I}$ of isometries acting transitively on $\mcl$, i.e.\ for any $\eta, \zeta \in \mcl$ there is an isometry $\psi:\mcl\to\mcl$ that satisfies $\psi(\eta)=\zeta$.
\end{defi}
Let $\mathcal{I}_O := \{\psi \in \mathcal{I} : \psi(O) = O\}$ be the stabiliser subgroup of $\mathcal{I}$ with respect to $O$. Then, for each $\eta\in\mcl$, we may choose a representative $\psi_\eta \in \mathcal{I} / \mathcal{I}_O$, such that $\psi_\eta(O) = \eta$ (cf.\ \cite[Sec.\ 5]{CohenLifshits2012}). Whenever we talk about homogeneous Riemannian manifolds we will assume that such a choice of representatives has been made. We give some examples for the standard model spaces $\R^d, ~\sph_d$ and $\hyp_d$, where we it is convenient to describe the isometries as translations along geodesic paths.

\begin{expl}\label{expl_shortestDistancesAndIsometriesModelSpaces}
    \begin{enumerate}
        \item In Euclidean space geodesic paths are given by straight lines. So if one wants to send point $0$ to point $t\in\R^d$ one may translate the whole space along this straight line. 
        The corresponding isometry is just the translation $\psi_t(s):=s+t$ for any $s\in\R^d$.
        
        \item \label{expl_shortestDistancesAndIsometriesModelSpaces_sph} If we take the usual representation of the sphere $\sph_d$ in $\R^{d+1}$ then we obtain the geodesic paths as the intersections of the sphere $\sph_d$ and a $2$-dimensional plane through the origin in $\R^{d+1}$. Translations along straight lines are given by rotations along a single axis.
        
        \item The usual representation of the hyperboloid in $\R^{d+1}$ is the smooth ``cone-like'' structure given by $\hyp_d := \{ \eta\in\R^{d+1} : \eta_1 > 0, ~ \eta\circ \eta = -1\}$, where $\eta\circ \zeta := - \eta_1 \zeta_1 + \eta_2 \zeta_2 + \ldots + \eta_{d+1} \zeta_{d+1}$. The geodesic paths on the hyperboloid are given by the intersection of $\hyp_d$ with any $2$-dimensional plane through the origin in $\R^{d+1}$ and analogously to the sphere there are the \emph{hyperbolic rotations} (Lorentz transforms) that rotate the space around a single axis.
    \end{enumerate}
\end{expl}

\begin{rem}\label{rem_isometricActionOnBalls}
    Since $\psi_\eta$ is an isometry for any $\eta\in\mcl$, we have that for every $r>0$
    \begin{align*}
        \psi_\eta(\bb_r(O)) = \bb_r(\eta), \qquad \text{and} \qquad \psi_\eta^{-1}(\bb_r(O)) = \bb_r(\psi_\eta^{-1}(O)).
    \end{align*}
    These identities will be important in the main argument of this section.
\end{rem}

We want to define stationary increments in a generalised sense. There is a \emph{weak} version, which can be found in \cite[Def.\ 5.2]{Istas2012} and which may be applied in a very general sense, and a \emph{strong} version in \cite[Sec.\ 5.4]{CohenLifshits2012}, which requires the homogeneous structure of the underlying manifold. Unless explicitly stated otherwise we always refer to \emph{strong} stationary increments. We use the notion of weak stationary increments further below to generalise the concept of self-similarity, cf.\ Definition~\ref{def_selfSimilarRiemann}.

\begin{defi}\label{def_stationaryIncr}
    A real valued field $(X_\eta)_{\eta\in\mcl}$ indexed by a homogeneous Riemannian manifold $\mcl$ is said to possess \emph{strong} stationary increments if for all $\zeta\in\mcl$ we have
    \begin{align*}
        (X_\eta - X_\zeta)_{\eta\in\mcl} \eqd (X_{\psi_\zeta^{-1}(\eta)})_{\eta\in\mcl},
    \end{align*}
    where the equality is in finite dimensional distribution.

    A real valued field $(X_\eta)_{\eta\in\mcl}$ indexed by a Riemannian manifold $\mcl$ is said to possess \emph{weak} stationary increments if for all $\eta,\zeta\in\mcl$ the distribution of the increments is a function of the distance $d(\eta,\zeta)$, i.e.\
    \begin{align*}
        \expec{e^{i \lambda (X_\eta - X_\zeta)}} = f(\lambda, d(\eta,\zeta))
    \end{align*}
    for some function $f:\R\times\R_{\geq 0} \to \C$, any $\eta,\zeta\in\mcl$ and $\lambda\in\R$.
\end{defi}

\begin{rem}\label{rem_OnWeakStationarity}
    \begin{enumerate}
        \item 
        Every fractional L\'evy field on a (not necessarily homogeneous) Rie\-mann\-ian manifold $\mcl$ is weak stationary. 
        \item Note that despite the terminology of ``weak'' and ``strong'', which may be found alongside each other for $\R^d$, cf.\ \cite[Def.\ 5.2]{Istas2012} and \cite[Eq.\ (3.4)]{Istas2012}, neither one implies the other, cf.\ Example~\ref{expl_weakStrongStationaryIncrements} (b) and (c).
    \end{enumerate}
\end{rem}

We provide some examples.

\begin{expl}\label{expl_weakStrongStationaryIncrements}
    \begin{enumerate}
        \item Let $\mcl=\R^d$. Then strong stationary increments in the sense of Definition \ref{def_stationaryIncr} are just the stationarity of increments in the usual sense if one picks the set of representatives given in Example \ref{expl_shortestDistancesAndIsometriesModelSpaces} (a). It is therefore a natural extension of the classical definition.
        \item Let $\mcl=\R$ and let the process $(X_t)_{t\in\R}$ be defined by $X_t\equiv Y$ for some $Y\sim\ncl(0,1)$ and all $t\in\R$. This process has weak stationary increments, but does not have strong stationary increments. 
        \item Let $\mcl=\R$, let $\psi_t(s):=t+s$ for any $s,t\in\R$ and let the process $(X_t)_{t\in\R}$ be defined by $X_t:=t$ for all $t\in\R$ has strong stationary increments, but does not have weak stationary increments.
    \end{enumerate}
\end{expl}

With this generalisation we can prove the following Theorem~\ref{thm_densityArgmaxRiemannian}. We will again impose the restriction that our stochastic process attains its maximum a.s.\ uniquely on compact sets, such that the argmax is a.s.\ well-defined. Note that the definition of the argmax
\begin{align*}
    \tau_{\ol{A}} := \argmax_{s \in \ol{A}} X_s
\end{align*}
for relatively compact sets $A\subseteq\mcl$ is the same as in the Euclidean case. We are going to show that this assumption is sufficient for our purposes in Corollary~\ref{cor_densityArgmaxHyperbolic}. We will elaborate on the other assumptions in Remark~\ref{rem_generalisationArgmaxRiemann}. 

\begin{thm}\label{thm_densityArgmaxRiemannian}
      Let $\mcl$ be a homogeneous Riemannian manifold where we choose some point $O\in\mcl$. Let $(X_\eta)_{\eta\in\mcl}$ be a real-valued separable stochastic process that has strong stationary increments in the sense of Definition~\ref{def_stationaryIncr} and that attains its maximum on any compact set a.s.\ at a unique point. Choose any $r'<\inj(O)$ and let $0< r < r'$ be given by Lemma~\ref{lem_ballVolumeBoundsGeneralRiemann} with $A:=\bb_{r'}(O)$. Then for any $0<T\leq r$ the random variable $\tau_{\olb_T(O)}$ has a density on $\bb_T(O)$ with respect to the Riemannian measure on $\mcl$.
\end{thm}
\begin{proof}
    Choose $0 < \delta < T/2$. We are going to show that for any $\eta\in\bb_{T-\delta}(O)$ and any $0< \eps < \delta/2$ the inequality
    \begin{align}\label{eq_argmaxDensityUpperBoundRiemann}
        \prob{\tau_{\olb_T(O)}\in \bb_\eps(\eta)} \leq \mathdutchbcal{c}_\mcl^{-1} ~ \eps^d ~ 2^d ~ \delta^{-d}
    \end{align}
    holds. Here, $\mathdutchbcal{c}_\mcl$ is the constant given in Lemma~\ref{lem_epsBallNumber_Riemann} with respect to $r$. The set $\bb_T(O)$ is relatively compact, since $T<\inj(O)$, which guarantees that the exponential map is a diffeomorphism and thus that the metric topology of the tangent space coincides with the metric topology on the manifold (cf.\ \cite[Cor.\ I.6.1]{Chavel2006}). 
    
    From \eqref{eq_argmaxDensityUpperBoundRiemann} we apply Lemma~\ref{lem_ballInequalityImpliesSurfaceDensity} with $A:=\bb_{T-\delta}(O)$, with $c_A:=\mathdutchbcal{c}_\mcl^{-1} ~ 2^d ~ \delta^{-d}$ and with $\eps_A:=\delta/2$, and conclude that the argmax has a density with respect to  the Riemannian measure on $\bb_{T-\delta}(O)$. Since we can choose $\delta$ arbitrarily small, we obtain the existence of a density on the whole $\bb_T(O)$.

    We are going to proceed by contradiction to prove \eqref{eq_argmaxDensityUpperBoundRiemann}. Suppose that there is a $\delta < T/2$, an $\eta\in\bb_{T-\delta}(O)$ and a $0< \eps < \delta/2$, such that the inequality in \eqref{eq_argmaxDensityUpperBoundRiemann} fails. According to Lemma~\ref{lem_epsBallNumber_Riemann}, we can find points $\xi_i$ with $i=1, \ldots, \lceil \cdu_\mcl (\delta/2)^d \eps^{-d} \rceil$, such that all balls $\bb_\eps(\xi_i)$ are disjoint subsets of $\bb_{\delta/2}(\eta)$. Let $(\psi_\eta)_{\eta\in\mcl}$ denote the representatives of the isometry group with $\psi_\eta(O)=\eta$ discussed in the beginning of this section. Then for all $i,j\in\{1, \ldots, \lceil \cdu_\mcl (\delta/2)^d \eps^{-d} \rceil\}$ we have
    \begin{align}\label{eq_ballInclusionRiemann}
        \bb_\eps(\xi_i) \subseteq \bb_T( \psi_{\xi_j} \circ \psi_\eta^{-1}(O)),
    \end{align}
    since for any $\zeta\in\bb_\eps(\xi_i)$ we can calculate using the defining properties of the isometries in the second and fourth equality, the metric triangle inequality in the first and second inequality, Remark~\ref{rem_isometricActionOnBalls} in the first and third equality, and the fact that $\zeta$ and $\xi_j$ are inside $\bb_{\delta/2}(\eta)$ in the last inequality that
    \begin{align*}
        d(\zeta, \psi_{\xi_j} \circ \psi_\eta^{-1}(O)) 
        & = d(\zeta, \xi_j) + d(\psi_{\xi_j}^{-1}(\xi_j), \psi_\eta^{-1}(O)) \\
        & = d(\zeta, \xi_j) + d(O, \psi_\eta^{-1}(O)) \\
        & = d(\zeta, \xi_j) + d(\eta, O) \\
        & \leq d(\zeta, \eta) + d(\eta, \xi_j) + d(\eta, O) \\
        & < \frac{\delta}{2} + \frac{\delta}{2} + (T-\delta) \\
        & = T.
    \end{align*}
    To obtain information about the distribution of the argmax we use the separability and the stationary increments of the field to obtain for any $\zeta, \theta\in\mcl$ that
    \begin{align}\label{eq_fundamentalArgmaxShift_Riemann}
        \tau_{\olb_T(\theta)} 
        &= \argmax_{\xi\in \olb_T(\theta)} (X_\xi) 
        = \argmax_{\xi\in \olb_T(\theta)} (X_\xi - X_\zeta) 
        \eqd \argmax_{\xi\in \olb_T(\theta)} (X_{\psi_\zeta^{-1}(\xi)}) \notag\\
        & = \psi_\zeta \left( \argmax_{\xi'\in \olb_T(\psi_\zeta^{-1}(\theta))} (X_{\xi'}) \right) 
        = \psi_\zeta \left( \tau_{\olb_T(\psi_\zeta^{-1}(\theta))} \right).
    \end{align}
    Setting $\theta:=\psi_\zeta(\theta')$ for any $\theta'\in\mcl$ we also obtain
    \begin{align}\label{eq_fundamentalArgmaxShift_Riemann_alt}
        \psi_\zeta^{-1}\left(\tau_{\olb_T(\psi_\zeta(\theta'))}\right) 
        \eqd \tau_{\olb_T(\theta')}.
    \end{align}
    
    Therefore, using Remark~\ref{rem_isometricActionOnBalls}, equation~\eqref{eq_fundamentalArgmaxShift_Riemann_alt} with $\zeta:=\xi_i$ and $\theta':=\psi_\eta^{-1}(O)$, and then \eqref{eq_fundamentalArgmaxShift_Riemann} with $\zeta:=\eta$ and $\theta:=O$ we may infer that
    \begin{align}\label{eq_allProbEqualArgmax_Riemann}
        \prob{\tau_{\olb_T(\psi_{\xi_i} \circ \psi_\eta^{-1}(O))}\in \bb_\eps(\xi_i)} 
        & = \prob{\tau_{\olb_T(\psi_{\xi_i} \circ \psi_\eta^{-1}(O))}\in \psi_{\xi_i} \left(\bb_\eps(O)\right)} \notag \\
        & = \prob{\psi_{\xi_i}^{-1}\left(\tau_{\olb_T(\psi_{\xi_i} \circ \psi_\eta^{-1}(O))}\right) \in \bb_\eps(O)} \notag \\
        & = \prob{\tau_{\olb_T(\psi_\eta^{-1}(O))} \in \bb_\eps(O)} \notag \\
        & = \prob{\psi_\eta \left(\tau_{\olb_T(\psi_\eta^{-1}(O))} \right) \in \psi_\eta \left(\bb_\eps(O)\right)} \notag \\
        & = \prob{\tau_{\olb_T(O)}\in \bb_\eps(\eta)},
    \end{align}
    for each $\xi_i$. Thus, we have shown that the events $\{\tau_{\olb_T(\psi_{\xi_i} \circ \psi_\eta^{-1}(O))}\in \bb_\eps(\xi_i)\}$ are of the same probability for all indices $i$. Furthermore, for $i\neq j$, we will check that these events may only intersect on a set of measure zero. For this purpose, define the events
    \begin{align*}
        \Omega_{\xi_i, \xi_j} & := \{ \tau_{\olb_T( \psi_{\xi_i} \circ \psi_\eta^{-1}(O))}\in \bb_\eps(\xi_i), ~\tau_{\olb_T( \psi_{\xi_j} \circ \psi_\eta^{-1}(O))}\in \bb_\eps(\xi_j) \}.
    \end{align*}
    
    We show that this is a set of measure zero for all $i\neq j$. The set
    \begin{align}\label{eq_OmegaXiiXijRiemann}
        \Omega_{\xi_i, \xi_j} \cap \{ X_{\tau_{\olb_T( \psi_{\xi_i} \circ \psi_\eta^{-1}(O))}} < X_{\tau_{\olb_T( \psi_{\xi_j} \circ \psi_\eta^{-1}(O))}} \}
    \end{align}
    must be empty, since by the definition of $\Omega_{\xi_i, \xi_j}$ and by \eqref{eq_ballInclusionRiemann} we would have
    \begin{align*}
        \tau_{\olb_T( \psi_{\xi_j} \circ \psi_\eta^{-1}(O))} \in \bb_\eps(\xi_j) \subseteq \olb_T( \psi_{\xi_i} \circ \psi_\eta^{-1}(O)),
    \end{align*}
    which would imply that $X_{\tau_{\olb_T( \psi_{\xi_i} \circ \psi_\eta^{-1}(O))}} \geq X_{\tau_{\olb_T( \psi_{\xi_j} \circ \psi_\eta^{-1}(O))}}$ contradicting the condition in the second set in \eqref{eq_OmegaXiiXijRiemann}. For the same reason, $\Omega_{\xi_i, \xi_j}$ intersecting the event $X_{\tau_{\olb_T( \psi_{\xi_i} \circ \psi_\eta^{-1}(O))}} > X_{\tau_{\olb_T( \psi_{\xi_j} \circ \psi_\eta^{-1}(O))}}$
    must be the empty set, as well. Using the a.s.\ uniqueness of the maximum, the intersection of $\Omega_{\xi_i, \xi_j}$ with the event $X_{\tau_{\olb_T( \psi_{\xi_i} \circ \psi_\eta^{-1}(O))}} = X_{\tau_{\olb_T( \psi_{\xi_j} \circ \psi_\eta^{-1}(O))}}$
    is a set of measure zero. Thus, we conclude that $\Omega_{\xi_i,\xi_j}$ is a set of measure zero for all $i\neq j$.

    Finally, we apply our knowledge about the sets $\Omega_{\xi_i,\xi_j}$, our assumption that \eqref{eq_argmaxDensityUpperBoundRiemann} fails, \eqref{eq_allProbEqualArgmax_Riemann}, and the assumptions about our constants to obtain that
    \begin{align*}
        1 & \geq \prob{\bigcup_{i=1}^{\lceil \cdu_\mcl (\delta/2)^d \eps^{-d} \rceil} \{ \tau_{\olb_T(\psi_{\xi_i} \circ \psi_\eta^{-1}(O))}\in \bb_\eps(\xi_i) \} } \\
        & = \sum_{i=1}^{\lceil \cdu_\mcl (\delta/2)^d \eps^{-d} \rceil} \prob{ \tau_{\olb_T(\psi_{\xi_i} \circ \psi_\eta^{-1}(O))}\in \bb_\eps(\xi_i) } \\
        & = \sum_{i=1}^{\lceil \cdu_\mcl (\delta/2)^d \eps^{-d} \rceil} \prob{ \tau_{\olb_T(O)} \in \bb_\eps(\eta) } \\
        & > \cdu_\mcl ~ (\delta/2)^d ~ \eps^{-d} ~ \cdu_\mcl^{-1} ~ \eps^d ~ 2^d ~ \delta^{-d} \\
        & =  1,
    \end{align*}
    which is a contradiction. Thus, \eqref{eq_argmaxDensityUpperBound} holds and the theorem is proved as described in the beginning.
\end{proof}

\begin{rem}\label{rem_generalisationArgmaxRiemann}
    (a) The generalisation from Remark~\ref{rem_generalisationArgmax} (a) to stationary instead of stationary increment fields holds completely analogously.
    
    (b) The generalisation from Remark~\ref{rem_generalisationArgmax} (b) to non-unique $\argmax$ cannot be applied analogously, since it is not clear that for $\eta,\zeta,\xi\in \olb_T(O)$ and some total order `$<$' we would have that $\eta < \zeta$ implies $\psi_\xi(\eta) < \psi_\xi(\zeta)$.
\end{rem}

\begin{rem}\label{rem_fundamentalArgmaxShiftSeparate}
    The identities \eqref{eq_fundamentalArgmaxShift_Riemann} and \eqref{eq_fundamentalArgmaxShift_Riemann_alt} are not specific to the proof and hold in general, i.e.\ if $\mcl$ is a homogeneous Riemannian manifold and $(X_\eta)_{\eta\in\mcl}$ is a separable process with stationary increments then for any $\theta, \theta', \zeta, \zeta'\in\mcl$ and $0<T<\inj(\theta)$, $0<T'<\inj(\theta')$ we have
    \begin{align*}
        \tau_{\olb_T(\theta)} 
        \eqd \psi_\zeta \left( \tau_{\olb_T(\psi_\zeta^{-1}(\theta))} \right) \qquad \text{and} \qquad
        \psi_{\zeta'}^{-1} \left(\tau_{\olb_{T'}(\psi_{\zeta'}(\theta'))}\right) 
        \eqd \tau_{\olb_{T'}(\theta')}.
    \end{align*}
\end{rem}

We now want to apply Theorem~\ref{thm_densityArgmaxRiemannian} to fractional L\'evy fields and derive the existence of densities for the argmax. Note that the hyperbolic space is a homogeneous Riemannian manifold (cf.\ Example~\ref{expl_shortestDistancesAndIsometriesModelSpaces}).

We first show that fractional L\'evy fields are a.s.\ continuous in general. This can be done using the usual Kolmogorov-Chentsov approach resulting in Hölder continuity, for which there is a generalisation available in \cite{KraetschmerUrusov2022}. Since we only require a.s.\ continuity and the calculation yields a bound for the expected supremum, which we need later on, we check the classical conditions of Dudley's theorem.

The following is a combination of \cite[Thm.\ 1.3.3]{AdlerTaylor2009} and \cite[Thm.\ 1.3.5]{AdlerTaylor2009}. We modified the notation to suit our purposes.
\begin{lem}[Dudley's theorem]\label{lem_MetricSupEstimate}
    Let $(X_\eta)_{\eta\in\mathfrak{M}}$ be a centred Gaussian field on a compact metric space $\mathfrak{M}$ with \emph{canonical} metric $d_X(\eta,\zeta) := \expec{(X_\eta-X_\zeta)^2}^{1/2}$.
    Let $N_X(x)$ be the smallest number of canonical balls of radius $x$ that cover $\mathfrak{M}$. Then
    \begin{align*}
        \expec{\sup_{\eta\in\mathfrak{M}} X_\eta} \leq K \int_0^{\diam_X(\mathfrak{M})/2} \sqrt{\log N_X(x)} \dd x,
    \end{align*}
    where $K$ is a universal constant and $\diam_X(\mathfrak{M}):=\sup_{\eta,\zeta\in \mathfrak{M}} d_X(\eta,\zeta)$ is the diameter with respect to the canonical metric. Moreover, if there is an $\eps_0 >0$ so that
    \begin{align}\label{eq_dudleyContinuityCondition}
        \int_0^{\eps_0} \sqrt{\log N_X(x)} \dd x < \infty
    \end{align}
    then $(X_\eta)_{\eta\in\mathfrak{M}}$ is a.s.\ continuous.
\end{lem}

We apply Lemma \ref{lem_MetricSupEstimate} to general fractional L\'evy fields.

\begin{lem}\label{lem_epsBallSupEstimate}
    Let $\mcl$ be a homogeneous Riemannian manifold with some fixed point $O\in\mcl$ on which the fractional L\'evy field exists for a Hurst parameter $H>0$. Then it is a.s.\ continuous and there exists a constant $C_1 > 0$ (depending only on $d$ and $H$), such that for all $\zeta\in\mcl$ and all $0 < \eps < 1$ we have
    \begin{align*}
        \expec{\sup_{\eta\in\olb_{\eps}(\zeta)} X_\eta} \leq \eps^H \sqrt{- C_1 \log \eps} + \mathcal{O}\left(\frac{\eps^H}{\sqrt{-\log \eps}}\right).
    \end{align*}
\end{lem}
\begin{proof}
    To apply Lemma \ref{lem_MetricSupEstimate} we first need $\diam_X(\olb_\eps(\zeta))$. To this end, we note that
    \begin{align*}
        d_X(\eta,\zeta) = \expec{(X_\eta-X_\zeta)^2}^{1/2} = d(\eta,\zeta)^H,
    \end{align*}
    where $d$ is the shortest distance metric on $\mcl$, as usual. Therefore,
    \begin{align*}
        \diam_X(\olb_\eps(\zeta)) \leq (2\eps)^H.
    \end{align*}
    Since $\bb_x^{(\can)}(\zeta) = \bb_{x^{1/H}}(\zeta)$ and since we can infer from e.g. \cite[Prop. 3.1 (i)]{KraetschmerUrusov2022} that there exists a constant $c > 0$, such that $N(x) \leq c ~ x^{-d}$ for all $0<x<1$, where $N(x)$ is the number of geodesic balls needed to cover $\olb_x(\zeta)$, we may conclude the bound
    \begin{align*}
        N_X(x) \leq c ~ x^{-d/H}.
    \end{align*}
    For $0<\eps<1$, we obtain
    \begin{align*}
        \int_0^\eps \sqrt{-\log x} ~ \dd x & = \int_{\sqrt{-\log \eps}}^\infty x \left(2x e^{-x^2}\right) \dd x \notag \\
        & = \eps \sqrt{-\log \eps} + \int_{\sqrt{-\log \eps}}^\infty e^{-x^2} \dd x \notag \\
        & \leq \eps \sqrt{-\log \eps} + \int_{\sqrt{-\log \eps}}^{\infty} \frac{x}{\sqrt{-\log \eps}} e^{-x^2} \dd x \notag \\
        & = \eps \sqrt{-\log \eps} + \frac{\eps}{2 \sqrt{-\log \eps}}.
    \end{align*}
    
    Combining the previous estimates and Lemma~\ref{lem_MetricSupEstimate}, we obtain that
    \begin{align}\label{eq_supremumBoundEntropy}
        \expec{\sup_{\eta\in\olb_{\eps}(\zeta)} X_\eta} 
        & \leq K \int_0^{\diam_X(\olb_\eps(\zeta))/2} \sqrt{\log N_X(x)} \dd x \notag \\
        & \leq K \int_0^{2^{H-1}\eps^H} \sqrt{\log\left( c ~ x^{-d/H} \right)  } \dd x \notag \\
        & = K c^{H/d} \sqrt{\frac{d}{H}} \int_0^{2^{H-1} c^{-H/d} \eps^H} \sqrt{ -\log\left( x \right)  } \dd x \notag \\
        & \leq \eps^H \sqrt{- C_1 \log \eps} + \mathcal{O}\left(\frac{\eps^H}{\sqrt{-\log \eps}}\right),
    \end{align}
    for some constant $C_1 > 0$ depending on $H$ and $d$, which the claim. 
    
    We finally prove continuity: If $\mathfrak{M}$ is any relatively compact subset of $\mcl$, then there is a constant $c>0$ so that the estimate $N_X(x) \leq c ~ x^{-d/H}$ holds, cf.\ \cite[Prop.\ 3.1]{KraetschmerUrusov2022}. Thus, by the same calculation as in \eqref{eq_supremumBoundEntropy} the continuity condition \eqref{eq_dudleyContinuityCondition} of Lemma~\ref{lem_MetricSupEstimate} is satisfied. Thus the process $(X_\eta)_{\eta\in\mcl}$ is a.s.\ continuous on any compact set in $\mcl$ and since $\mcl$ may be covered by countably many  relatively compact sets by our general assumptions in Section~\ref{sec_RiemannianBackground} we obtain that the process is a.s.\ continuous on the whole $\mcl$.
\end{proof}

\begin{cor}\label{cor_densityArgmaxHyperbolic}
    Let $\mcl$ be a homogeneous Riemannian manifold with some fixed point $O\in\mcl$ on which the fractional L\'evy field exists for a Hurst parameter $H>0$. Choose any $r'<\inj(O)$ and let $0< r < r'$ be given by Lemma~\ref{lem_ballVolumeBoundsGeneralRiemann} with $A:=\bb_{r'}(O)$. Then for any $0<T\leq r$ the $\argmax$ function $\tau_{\olb_T(O)}$ of the fractional L\'evy field has density with respect to  Riemannian measure in $\bb_T(O)$.
\end{cor}
\begin{proof}    
    Theorem \ref{thm_densityArgmaxRiemannian} implies the statement if we can show that L\'evy fields have (strong) stationary increments in the sense of Definition \ref{def_stationaryIncr} and by showing that the $\argmax$ is a.s.\ unique.

    To show that the field attains an a.s.\ unique maximum in $\olb_T(O)$ we apply Lemma~\ref{lem_uniqueMaximum}: The field is a.s.\ continuous on any relatively compact subset of $\mcl$ by Lemma~\ref{lem_epsBallSupEstimate} and thus, in particular, on $\olb_T(O)$.
    The set $\olb_T(O)$ is compact and therefore trivially $\sigma$-compact. This also implies that the field attains its maximum in $\olb_T(O)$ in at least one point. Using the fact that the field is centred we see that $\var{X_\eta - X_\zeta} = \expec{(X_\eta-X_\zeta)^2} = d(\eta, \zeta)^{2H} \neq 0$ for any $\eta \neq \zeta$ in $\mcl$. Thus, the assumptions of Lemma~\ref{lem_uniqueMaximum} are satisfied and we obtain that the maximum is a.s.\ attained in no more than one point, i.e.\ together with the compactness of $\olb_T(O)$ a.s.\ in exactly one point.
    
    For the strong stationarity of increments we calculate the covariance 
    \begin{align*}
        & \expec{(X_H(\eta_1) - X_H(\zeta))(X_H(\eta_2) - X_H(\zeta))} \\
        & = \eh \Big( d(\eta_1,O)^{2H} + d(\eta_2,O)^{2H} - d(\eta_1,\eta_2)^{2H} + 2 d(\zeta,O)^{2H} - d(\zeta,O)^{2H} - d(\eta_1,O)^{2H} + d(\zeta, \eta_1)^{2H} \\
        & \qquad - d(\zeta,O)^{2H} - d(\eta_2,O)^{2H} + d(\zeta, \eta_2)^{2H} \Big) \\
        & = \eh \left( d(\zeta,\eta_1)^{2H} + d(\zeta,\eta_2)^{2H} - d(\eta_1,\eta_2)^{2H} \right) \\
        & = \eh \left( d(O, \psi^{-1}_\zeta(\eta_1))^{2H} + d(O, \psi^{-1}_\zeta(\eta_2))^{2H} - d(\psi^{-1}_\zeta(\eta_1),\psi^{-1}_\zeta(\eta_2))^{2H} \right) \\
        & = \expec{X_H(\psi^{-1}_\zeta(\eta_1)) ~ X_H(\psi^{-1}_\zeta(\eta_2))},
    \end{align*}
    using the isometry property in the second to last step. We have now shown all requirements of Theorem~\ref{thm_densityArgmaxRiemannian}, and the statement follows.
\end{proof}

\section{Persistence probability of the hyperbolic fractional L\'evy field}\label{sec_persistenceHyperbolic}

The key insight from the previous chapter is the existence of a density. This is going to be in the core argument for the proof of the lower bound of the persistence probability. The upper bound will be handled using a comparison argument. The same argument was used in \cite{AurzadaHelmer2026} for the lower bound, but we will reintroduce it here to keep everything self-contained.

The bounds derived in Lemma~\ref{lem_hfbmLowerBound} and Lemma~\ref{lem_hfbmUpperBound} below prove Theorem~\ref{thm_mainHyperbolic}.

\subsection{Lower Bound}

We prove the lower bound by combining classical Gaussian process theory, in particular the Borell-TIS inequality, with the existence of a density similarly to \cite[Lem.\ 4]{Molchan99}. In the proof of Lemma~\ref{lem_hfbmLowerBound} below, however, we use a weaker assumption: It suffices that there are points of positive density sufficiently close to $O$. This condition will be checked using the following lemma.

\begin{lem}\label{lem_existenceDensityPoint}
    Let $A$ be an open relatively compact set in a $d$-dimensional Riemannian manifold and let $f: A \to \R$ be a density with respect to  the Riemannian measure of a random variable $X$ on $A$. If $\prob{X\in A} > 0$ then there is a point $\eta\in A$ and a constant $c>0$ with $\prob{X\in \bb_\eps(\eta)} > c ~ \eps^{d}$ for $\eps >0$ small enough.
\end{lem}
\begin{proof}
    The fact that $f$ is a density implies that for almost every point $\eta\in A$ we have
    \begin{align}\label{eq_existenceDensityPoint}
        \lim_{\eps \to 0} \frac{\prob{X\in \bb_\eps(\eta)}}{\sigma(\bb_\eps(\eta))} = f(\eta).
    \end{align}
    This can be inferred from the metric Lebesgue differentiation theorem in \cite[Thm.\ 3.4]{LucicPasqualetto2023}), in particular \cite[Eq.\ (3.3)]{LucicPasqualetto2023}, where the notation of the integral is explained in \cite[p.\ 4 of 51]{LucicPasqualetto2023} and where the requirements are satisfied due to \cite[Rem.\ 3.1]{LucicPasqualetto2023}, because $\R$ is separable and thus $f$ is separable valued. Since the volume $\sigma(\bb_\eps(\eta))$ scales with $\eps^d$ for any $\eta\in A$ (cf.\ Lemma~\ref{lem_ballVolumeBoundsGeneralRiemann}), there must be some point, at which $\prob{X\in \bb_\eps(\eta)}$ also scales with $\eps^d$, otherwise we obtain a contradiction through the assumption that $\prob{X\in A} > 0$.
\end{proof}

Furthermore, we require some classical tools from Gaussian random field theory. The following well-known theorem can be found for example in \cite[Theorem 2.1.1]{AdlerTaylor2009}. We have already proven its assumption to be satisfied for any fractional L\'evy field above in Lemma~\ref{lem_epsBallSupEstimate}.

\begin{lem}[Borell-TIS inequality]\label{lem_borellTIS}
    Let $\mathfrak{M}$ be a metric space and $(X_\eta)_{\eta\in\mathfrak{M}}$ a real-valued centred Gaussian process that is a.s.\ bounded on $\mathfrak{M}$. Then $\expec{\sup_{\eta\in\mathfrak{M}} X_\eta} < \infty$
    and for all $u> \expec{\sup_{\eta\in\mathfrak{M}} X_\eta}$, we have (with $\sigma_\mathfrak{M}^2 := \sup_{\eta\in\mathfrak{M}} \expec{X_\eta^2}$)
    \begin{align*}
        \prob{\sup_{\eta\in\mathfrak{M}} X_\eta > u} \leq \exp\left(-\left(u-\expec{\sup_{\eta\in\mathfrak{M}} X_\eta}\right)^2 / (2\sigma_\mathfrak{M}^2) \right).
    \end{align*}
\end{lem}

The following lemma allows us to conclude a certain type of non-degeneracy of the argmax shown in the corollary afterwards. The proof may be carried out similarly to the proof of \cite[Lemma A.1]{AurzadaHelmer2026}.

\begin{lem}\label{lem_NonPersistenceOnDomainsWithoutZero} 
    Let $(X_H(\eta))_{\eta\in\hyp_{d}}$ be the hyperbolic fractional L\'evy field and let $A$ be a compact subset of $\hyp_{d}$, such that $\inf_{\eta\in A} \expec{X_H(\eta)^2} > 0$. Then with positive probability the field is non-positive on $A$, i.e.
    \begin{align*}
        \prob{X_H(\eta) \leq 0 ~ \forall \eta\in A} > 0.
    \end{align*}
\end{lem}

The following corollary provides an essential non-degeneracy condition for the main proof of this section.

\begin{cor}\label{cor_nonDegeneracyArgmax}
    The argmax for the hyperbolic fractional L\'evy field is non-degenerate in the sense that for any $\delta > 0$ we have that
    \begin{align*}
        \prob{\tau_{\olb_{T}(O)} \in \bb_\delta(O)} > 0.
    \end{align*}
\end{cor}
\begin{proof}
    An application of Lemma \ref{lem_NonPersistenceOnDomainsWithoutZero} yields that there is a positive probability for the process to be less than zero anywhere in $\olb_{T}(O)$ outside of a geodesic $\delta$-ball around $O$, i.e.\
    \begin{align*}
        \prob{X_H(\eta) \leq 0 ~ \forall \eta \in \olb_T(O)\setminus \bb_\delta(O) } > 0.
    \end{align*}
     Now take a piece of a geodesic line passing through $O$ inside of $\bb_\delta(O)$. By reparametrization this can be viewed as one-dimensional fractional Brownian motion with $H\leq 1/2$, so there is a.s.\ a point larger than zero on this line. This follows from the law of iterated logarithm for fractional Brownian motion on the real line, cf.\ \cite{Peyre2017} or \cite[Cor.\ 3.1]{Arcones1995}. Thus, the probability that the maximum lies in $\bb_\delta(O)$ is bounded from below by the probability inferred from Lemma \ref{lem_NonPersistenceOnDomainsWithoutZero}, i.e.\
     \begin{gather*}
         \prob{\tau_{\olb_{T}(O)} \in \bb_\delta(O) } \geq \prob{X_H(\eta) \leq 0 ~ \forall \eta \in \olb_T(O)\setminus \bb_\delta(O) } > 0.\qedhere
     \end{gather*}
\end{proof}

Now we can prove the following lower bound, which follows the steps of \cite[Lemma 4]{Molchan99} closely.

\begin{lem}\label{lem_hfbmLowerBound}
    Let $R>0$ be an arbitrary finite radius. For the hyperbolic fractional L\'evy field on $\hyp_{d}$ there exists a constant $C_2$, such that, for all $\eps>0$ small enough,
    \begin{align*}
        \prob{\sup_{\eta\in\bb_R(O)} X_H(\eta) < \eps} \geq C_2 ~ \eps^{\frac{d}{H}} ~ \left(\sqrt{-\log \eps}\right)^{-\frac{d}{H}}.
    \end{align*}
\end{lem}
\begin{proof}
    We are going to use the argmax of the hyperbolic fractional L\'evy field to bound the maximum from below. 
    We start by fixing some constant $c > 0$ to be chosen later and calculate
    \begin{align}\label{eq_mainLowerBound1}
        & \prob{ \tau_{\olb_R(O)} \in \bb_{\eps^{1/H}}(O)} \notag\\
        & = \prob{\sup_{\eta\in\bb_R(O)} X_H(\eta) < \eps \sqrt{- c \log \eps}, ~ \tau_{\olb_R(O)} \in \bb_{\eps^{1/H}}(O)} \notag\\
        & \qquad + \prob{\sup_{\eta\in\bb_R(O)} X_H(\eta) \geq \eps \sqrt{- c \log \eps}, ~ \tau_{\olb_R(O)} \in \bb_{\eps^{1/H}}(O)} \notag \\
        & \leq \prob{\sup_{\eta\in\bb_R(O)} X_H(\eta) < \eps \sqrt{- c \log \eps}} + \prob{\sup_{\eta\in\bb_R(O)} X_H(\eta) \geq \eps \sqrt{- c \log \eps}, ~ \tau_{\olb_R(O)} \in \bb_{\eps^{1/H}}(O)}.
    \end{align}

    \textit{Step 1:} We are first going to prove that the second term at the end of \eqref{eq_mainLowerBound1} decays at a polynomial rate depending on $c$. To this end, we apply Lemma \ref{lem_borellTIS} to obtain the estimate
    \begin{align}\label{eq_mainLowerRemainder1}
        & \prob{\sup_{\eta\in\bb_R(O)} X_H(\eta) > \eps \sqrt{- c \log \eps}, ~ \tau_{\olb_R(O)} \in \bb_{\eps^{1/H}}(O)}
        \leq \prob{\sup_{\eta\in\olb_{\eps^{1/H}}(O)} X_H(\eta) > \eps \sqrt{- c \log \eps}} \notag \\
        & \leq \exp\left(-\left(\eps \sqrt{- c \log \eps} - \expec{\sup_{\eta\in\olb_{\eps^{1/H}}(O)} X_H(\eta)}\right)^2 / (2\sigma_{\olb_{\eps^{1/H}}(O)}^2)\right),
    \end{align}
    which holds as long as $\eps \sqrt{- c \log \eps} > \expec{\sup_{\eta\in\olb_{\eps^{1/H}}(O)} X_H(\eta)}$ is true. We can check this using Lemma \ref{lem_epsBallSupEstimate}. Choosing $c>C_1/H$ we can calculate that
    \begin{align*}
        \eps \sqrt{- c \log \eps} - \expec{\sup_{\eta\in\olb_{\eps^{1/H}}(O)} X_H(\eta)}
        & \geq \eps \sqrt{- c \log \eps} - \eps \sqrt{- C_1/H \log \eps} + \ocl\left(\frac{\eps}{\sqrt{-\log \eps}}\right) \\
        & = \eps \sqrt{- \left(\sqrt{c} - \sqrt{C_1/H}\right)^2 \log \eps} + \ocl\left(\frac{\eps}{\sqrt{-\log \eps}}\right),
    \end{align*}
    which is greater than zero for all $\eps$ small enough and thus justifies the application of the Borell-TIS inequality in \eqref{eq_mainLowerRemainder1}. Noting that
    \begin{align*}
        \sigma_{\olb_{\eps^{1/H}}(O)}^2 = \sup_{\eta\in\olb_{\eps^{1/H}}(O)} \expec{X_H(\eta)^2} = \eps^2,
    \end{align*}
    we can continue from \eqref{eq_mainLowerRemainder1} and infer that
    \begin{align*}
        & \exp\left(-\left(\eps \sqrt{- c \log \eps} - \expec{\sup_{\eta\in\olb_{\eps^{1/H}}(O)} X_H(\eta)}\right)^2 / (2\sigma_{\olb_{\eps^{1/H}}(O)}^2)\right) \\
        & \leq \exp\left(-\left( \eps \sqrt{- \left(\sqrt{c} - \sqrt{C_1/H}\right)^2 \log \eps} + \mathcal{O}\left(\frac{\eps}{\sqrt{-\log \eps}}\right) \right)^2 / (2\eps^2)\right) \\
        & = \exp\left(-\left( \sqrt{- \left( \frac{\sqrt{c} - \sqrt{C_1/H}}{\sqrt{2}}\right)^2 \log \eps} + \mathcal{O}\left(\frac{1}{\sqrt{-\log \eps}}\right) \right)^2 \right) \\
        & = \exp\left(\left( \frac{\sqrt{c} - \sqrt{C_1/H}}{\sqrt{2}}\right)^2 \log \eps + \mathcal{O}(1) + \mathcal{O}\left(\frac{1}{-\log \eps}\right) \right)\in \ocl\left(\eps^{\left( \frac{\sqrt{c} - \sqrt{C_1/H}}{\sqrt{2}}\right)^2}\right),
    \end{align*}
    which proves that, indeed, $\prob{\sup_{\eta\in\bb_R(O)} X_H(\eta) > \eps \sqrt{- c \log \eps}, ~ \tau_{\olb_R(O)} < \eps^{1/H}}$ can be made to decay as arbitrarily fast polynomially in $\eps$, since $c>0$ can be made arbitrarily large.

    \textit{Step 2:} We are now going to show a lower bound of $\prob{\tau_{\olb_R(O)} < \eps^{1/H}}$. First, choose some $0<\delta < R$. Since $\inj(\hyp_d)=\inj(\mcl_{-1})=\infty$ (cf.\ \eqref{eq_modelSpace_injRadius}) and since any radius $r>0$ from Lemma~\ref{lem_ballVolumeBoundsGeneralRiemann} is admissible, we may choose $T:=R+\delta$ and apply Corollary~\ref{cor_densityArgmaxHyperbolic}. This yields that $\tau_{\olb_{R+\delta}(O)}$ possesses a Riemannian density in $\bb_{R+\delta}(O)$. Corollary~\ref{cor_nonDegeneracyArgmax} shows that there is actually mass of the $\argmax$ distribution in $\bb_{\delta}(O)$ and so by Lemma~\ref{lem_existenceDensityPoint} there exists a point of positive density $\xi\in\bb_\delta(O)$, i.e.\
    \begin{align*}
        \prob{\tau_{\olb_{R+\delta}(O)} \in \bb_{\eps^{1/H}}(\xi)} \geq \ol{c} ~ \eps^{\frac{d}{H}}
    \end{align*}
    for some constant $\ol{c} > 0$ and $\eps > 0$ small enough. Using the stationarity of increments, we may apply Remark~\ref{rem_fundamentalArgmaxShiftSeparate} with $\theta':= O$, $\zeta':= \xi$ and $T':=R$ to obtain that
    \begin{align*}
        \prob{\tau_{\olb_R(O)} \in \bb_{\eps^{1/H}}(O)}
        & = \prob{\psi_\xi^{-1}\left(\tau_{\olb_R(\psi_\xi(O))}\right) \in \bb_{\eps^{1/H}}(O)}
        = \prob{\tau_{\olb_R(\xi)} \in \bb_{\eps^{1/H}}(\xi)} \\
        & \geq \prob{\tau_{\olb_{R+\delta}(O)} \in \bb_{\eps^{1/H}}(\xi)}
        \geq \ol{c} ~ \eps^{\frac{d}{H}}.
    \end{align*}
    Finally, the result is obtained by plugging in the results from Step 1 \& Step 2 into \eqref{eq_mainLowerBound1} and by substituting $\eps':= \eps \sqrt{-c \log\eps}$.
\end{proof}

\subsection{Upper bound}\label{sec_upperboundHFBM}

The proof of the upper bound will be carried out as a comparison argument combining Slepian's Lemma from stochastics and Toponogov's theorem from geometry. The idea is the same as in \cite[Sec.\ 5.1 \& 5.2]{AurzadaHelmer2026}, which also offers a heuristic description of the geometric tool involved. Here, we present a more geometric approach.

Let $\exp_O:\R^d \to\hyp_d$ be the exponential map (cf.\ \eqref{eq_expMapDef}). This is a diffeomorphism between the tangent space of the hyperboloid at $O$, which may be identified with $\R^d$, and the hyperboloid itself. The map preserves lengths that are inside the injectivity radius and follow a straight / geodesic line from the origin (cf.\ \eqref{eq_expMapDef} and the explanation thereafter). Other lengths are not necessarily preserved, since the exponential map is not an isometry in general. By \eqref{eq_modelSpace_injRadius}, the injectivity radius at every point of the hyperboloid is infinite. Thus, for $t\in\R^d$ we have
\begin{align*}
    \norm{t}_\mcl = d(O, \exp_O(t)),
\end{align*}
cf.\ \eqref{eq_expMapDistancePreserving}. Since we only need the Riemannian metric in a single tangent space, we can choose an orthonormal base, with respect to  which the scalar product becomes the usual Euclidean scalar product in this tangent space, i.e.\ with the usual Euclidean distance $\norm{.}$,
\begin{align}\label{eq_exponentialMapEqualities}
    \norm{t} = d(O, \exp_O(t)).
\end{align}
Since the exponential map is a diffeomorphism we can plug in $t:=\exp_O^{-1}(\eta)$ for any $\eta\in\hyp_d$ and obtain that
\begin{align*}
    \norm{\exp^{-1}_O(\eta)} = d(O, \eta).
\end{align*}

Thus, if we now have two points $\eta,\zeta\in\hyp_d$, we know that $\exp_O^{-1}$ maps them to Euclidean space while keeping the distance to the respective \emph{origins} the same. We are now interested in how the distance $d(\eta,\zeta)$ changes through this mapping.

To this end, we may apply the following simplified version of the classical Toponogov (hinge) comparison theorem (cf.\ \cite[Thm.\ 12.2.2]{Petersen2016} or \cite[Thm.\ IX.5.1]{Chavel2006}).

\begin{lem}[Toponogov's Theorem]\label{lem_Toponogov}
    For any two points $\eta,\zeta\in\hyp_{d}$ it is true that
    \begin{align*}
        d(\eta,\zeta) \geq \norm{\exp_O^{-1}(\eta) - \exp_O^{-1}(\zeta)}.
    \end{align*}
\end{lem}

Compared to \cite[Thm.\ 5.3]{AurzadaHelmer2026} note that this is the exact opposite inequality. 
For a generalised local version, cf.\ Lemma~\ref{lem_rauchToponogovComparisonEstimate}.

The importance to us is the fact that we can turn this geometric comparison inequality into a stochastic comparison inequality using the following well known lemma by Slepian (cf.\ \cite{Slepian62}). Its restatement in continuous time can be found, for example, in \cite[Lemma 1.2.5]{Baumgarten2013thesis}.

\begin{lem}[Slepian's Lemma]\label{lem_Slepian}
    Let $(E,d)$ be a metric space and let $(X_t)_{t\in E}$ and $(Y_t)_{t\in E}$ denote two real-valued, separable, centred Gaussian processes, such that
    \begin{align*}
        \expec{X_t^2} & = \expec{Y_t^2}, \quad \forall t\in E,
        \qquad\qquad
        \expec{X_s X_t}  \leq \expec{Y_s Y_t}, \quad \forall s,t\in E.
    \end{align*}
    Let $f: E\to \R$ be a measurable function that is continuous except for countably many points. Then
    \begin{align*}
        \prob{X_t \leq f(t) \quad \forall t\in E} \quad \leq \quad \prob{Y_t \leq f(t) \quad \forall t\in E}.
    \end{align*}
\end{lem}

An immediate corollary to Slepian's Lemma for positively correlated processes is the following, which we will need later. This is the same as \cite[Cor.\ 5.6]{AurzadaHelmer2026}.

\begin{cor}\label{cor_SlepianPosCorr}
    Let $(E,d)$ be a metric space and let $(Y_t)_{t\in E}$ denote a real-valued, separable, centred Gaussian process such that $\expec{Y_s Y_t} \geq 0$, $\forall s,t\in E$.
    Let $f: E\to \R$ be a measurable function that is continuous except for possibly countably many points. Let $E = A \cup B$. Then
    \begin{align*}
        \prob{Y_t \leq f(t) ~ \forall t\in E} \quad \geq \quad \prob{Y_t \leq f(t) ~ \forall t\in A} \cdot \prob{Y_t \leq f(t) ~ \forall t\in B}.
    \end{align*}
\end{cor}

Combining Slepian's Lemma and Toponogov's Theorem lets us conclude the upper bound for the persistence probability of the hyperbolic fractional L\'evy field using the result for the Euclidean case.

\begin{lem}\label{lem_hfbmUpperBound}
    Let $(X_H(\eta))_{\eta\in\hyp_{d}}$ be the hyperbolic fractional L\'evy field on $\hyp_{d}$ and let $R>0$. Then, for some constant $c>0$,
    \begin{align*}
        \prob{ \sup_{\eta\in\bb^{\hyp_d}_R(O)} X_H(\eta) < \eps } \leq \eps^{\frac{d}{H}} \exp( \sqrt{ - c \log \eps} )
    \end{align*}
    for some constant $c>0$ and $\eps>0$ small enough.
\end{lem}
\begin{proof}
    Due to \eqref{eq_exponentialMapEqualities} we know that for all $\eta\in\hyp_{d}$
    \begin{align*}
        \expec{X_H(\eta)^2} = d(\eta,O)^{2H} = \norm{\exp_O^{-1}(\eta)}^{2H} = \expec{B_H(\exp_O^{-1}(\eta))^2},
    \end{align*}
    where $(B_H(t))_{t\in\R^d}$ is standard fractional Brownian motion with Hurst parameter $H$ in $\R^d$. From Toponogov's theorem in the form of Lemma \ref{lem_Toponogov} we furthermore know that for all $\eta,\zeta\in\hyp_{d}$
    \begin{align*}
        \expec{X_H(\eta) X_H(\zeta)} 
        & = \eh \left( d(O,\eta)^{2H} + d(O,\zeta)^{2H} - d(\eta,\zeta)^{2H} \right)& \\
        & \leq \eh \left( \norm{\exp_O^{-1}(\eta)}^{2H} + \norm{\exp_O^{-1}(\zeta)}^{2H} - \norm{\exp_O^{-1}(\eta)-\exp_O^{-1}(\zeta)}^{2H} \right) & \\
        & = \expec{B_H(\exp_O^{-1}(\eta)) B_H(\exp_O^{-1}(\zeta))}.
    \end{align*}
    We may now apply Slepian's Lemma in the form of Lemma \ref{lem_Slepian} and use the fact that $\exp_O^{-1} : \bb_R^{\hyp_d}(O) \to \bb_R^{\R^d}(0)$ is a bijection to obtain
    \begin{align*}
        \prob{ \sup_{\eta\in\bb^{\hyp_d}_1(O)} X_H(\eta) < \eps }
        & \leq \prob{ \sup_{\eta\in\bb^{\hyp_d}_R(O)} B_H(\exp_O^{-1}(\eta)) < \eps } \\
        & = \prob{ \sup_{t\in\bb^{\R^d}_R(0)} B_H(t) < \eps } \\
        & \leq \eps^{\frac{d}{H}} \exp( \sqrt{ - c ~ \log \eps} ),
    \end{align*}
    with an application of the Euclidean result \cite[Lem.\ 5]{Molchan99}. 
\end{proof}

Now Lemma~\ref{lem_hfbmLowerBound} and Lemma~\ref{lem_hfbmUpperBound} prove Theorem~\ref{thm_mainHyperbolic}.

\section{Persistence exponent of fractional L\'evy fields on Riemannian manifolds}\label{sec_persistenceGeneralRiemann}

In this section we derive the proof of Theorem~\ref{thm_mainRiemannianManifold} by comparing the persistence probability of fractional Lévy fields indexed by general Riemannian manifolds to those indexed by model spaces. The bounds derived in Lemma~\ref{lem_riemannPersistenceUpperBound} and Lemma~\ref{lem_riemannPersistenceLowerBound} below imply Theorem~\ref{thm_mainRiemannianManifold}.

As in Section~\ref{sec_RiemannianBackground} we assume that the Riemannian manifold $\mcl$ considered will be $\mathcal{C}^\infty$, connected, Hausdorff, with a countable base and without boundary. Again we identify some specific point $O\in\mcl$. We are only going to consider fractional L\'evy fields with Hurst parameters $0< H \leq 1/2$, provided they exist.

In our previous application of Toponogov's theorem (Lemma~\ref{lem_Toponogov}) we were able to use the fact that we wanted to compare a model space with a specific other model space. If we want to compare distances on an arbitrary manifold with distances on an arbitrary model space, we need a statement with weaker assumptions. Such a weaker statement exists (cf.\ Lemma~\ref{lem_rauchToponogovComparisonEstimate} below), but it will only work locally. To this end, we construct the following function $\Psi_\kappa$ for $\kappa\in\R$.

We require some specific diffeomorphisms between $\mcl$ and certain model spaces $\mcl_\kappa$. To this end let $\kappa\in\R$ be arbitrary and define the canonical isomorphism $\mathfrak{I}_O: T_O\mcl \to T_O \mcl_\kappa$ that identifies the tangent space of $\mcl$ at $O$ with the tangent space of the model space $\mcl_\kappa$ of the corresponding dimension. Note that there are two different points $O$, but since in both cases they just represent the origin of the respective coordinate system we use the same notation. The canonical isomorphism exists for any $\mcl, O$ and any $\kappa$, but they affect how large the following radius may be chosen: For $0< r\leq\min\{\inj_\mcl(O), \inj_{\mcl_\kappa}(O)\}$ (recall that the injectivity radius of any point on any Riemannian manifold we consider is positive, cf.\ Section~\ref{sec_RiemannianBackground}), define $\Psi_\kappa:  \bb_r^{\mcl}(O) \to \bb_r^{\mcl_\kappa}(O)$ by
\begin{align}\label{eq_exponentialTransferMap}
    \Psi_\kappa(\eta) := \left(\exp_{\mcl_\kappa, O}\circ \mathfrak{I}_O \circ \exp^{-1}_{\mcl, O}\right)(\eta),
\end{align}
where the subindex on the exponential map denotes which space it is defined on. By construction, this is a bijection, and, since the exponential maps (inside their injectivity radius) and the isometry are diffeomorphisms, the concatenation is a diffeomorphism. By length preservation of the exponential map with respect to  the coordinate origin (cf.\ \eqref{eq_expMapDistancePreserving}) and by isometry, we have that for $\eta\in\mcl$
\begin{align}\label{eq_expMapDistancePreservingGeneralised}
    d_\mcl(O,\eta) = d_{\mcl_\kappa}(\Psi_\kappa(O),\Psi_\kappa(\eta)),
\end{align}
where $d_\mcl(.,.)$ and $d_{\mcl_\kappa}(.,.)$ are the shortest (geodesic) distance functions in the respective spaces. Since $\Psi_\kappa(O)$ is just the fixed origin in a different manifold (namely $\mcl_\kappa$), which is clear from context, we will write $O$ for $\Psi_\kappa(O)$, as well.

The following is a streamlined version of the geometric Rauch theorem (cf.\ \cite[Thm.\ IX.2.3]{Chavel2006} and \cite[Sec.\ 4]{Karcher1989}), which is a localised version of Toponogov's theorem. It would not hold on such a large class of Riemannian manifolds if one did not restrict to a sufficiently small domain. 

\begin{lem}\label{lem_rauchToponogovComparisonEstimate}
    Under the above assumptions on the manifold $\mcl$ there exists $r'>0$ and $\kappa', \kcl' \in\R$ with $\kappa'\leq\kcl'$, such that for any $\kappa\leq\kappa'$, any $\kcl'\leq\kcl$, any $0<r\leq\min\left\{r', \inj(\mcl_\kcl)\right\}$, and any $\eta,\zeta\in\bb_r^{\mcl}(O)$ we have
    \begin{align*}
        d_\kcl(\Psi_\kcl(\eta), \Psi_\kcl(\zeta)) \leq d(\eta, \zeta) \leq d_\kappa(\Psi_\kappa(\eta), \Psi_\kappa(\zeta)),
    \end{align*}
    where $d_\kcl(.,.)$ and $d_\kappa(.,.)$ are the shortest distance functions in the respective model spaces $\mcl_\kcl$ and $\mcl_\kappa$ and $\Psi_\kcl, \Psi_\kappa$ are the diffeomorphisms defined in \eqref{eq_exponentialTransferMap}.
\end{lem}

To obtain the persistence results, we require an analogue of Lemma \ref{lem_NonPersistenceOnDomainsWithoutZero} for arbitrary Riemannian manifolds:

\begin{lem}\label{lem_NonPersistenceOnDomainsWithoutZeroRiemann} 
    Let $(X_H(\eta))_{\eta\in\mcl}$ be a fractional L\'evy field with Hurst parameter $0<H\leq 1/2$ and let $A$ be a relatively compact subset of $\mcl$, such that $\inf_{\eta\in A} \expec{X_H(\eta)^2} > 0$. Then with positive probability the field is non-positive on $A$, i.e.
    \begin{align*}
        \prob{X_H(\eta) \leq 0 ~ \forall \eta\in A} > 0.
    \end{align*}
\end{lem}

\begin{rem}\label{rem_proofOfDomainsWithoutZeroAndPositiveCorrelations}
    The proof of this lemma can, again, be carried out analogously to the proof of \cite[Lemma A.1]{AurzadaHelmer2026}. Note that this works due to the closure of $A$ being compact and because the process is positively correlated everywhere, i.e.\ $\expec{X_H(\eta) X_H(\zeta)}\geq 0$ for all $\eta,\zeta\in\mcl$. This can be seen by an application of triangle inequality for $d$ and the fact that $(a+b)^{2H} \leq a^{2H} + b^{2H}$ for $a,b\geq 0$ and $2H \leq 1$.
\end{rem}

We now define an abstraction of the concept of \textit{self-similarity} or \textit{scale invariance} by scaling the Riemannian metric. The definition is modelled after \cite[Sec.\ 3.4.2]{Gelbaum2014} and slightly generalised to be applicable for weak stationary fields (recall Definition~\ref{def_stationaryIncr} and Remark~\ref{rem_OnWeakStationarity} (a)).

\begin{defi}\label{def_selfSimilarRiemann}
    Let $(X(\eta))_{\eta\in\mcl}$ be a weak stationary centred Gaussian field indexed by the Riemannian manifold $(M,\abrac{.,.}_\mcl)$ with Riemannian metric $\abrac{.,.}_\mcl$ so that there is a function $g:\R_{\geq 0}\to\R_{\geq 0}$ which determines the structure function of the process, i.e.\ $\expec{(X(\eta)-X(\zeta))^2}=g(d_\mcl(\eta,\zeta))$ for all $\eta,\zeta\in\mcl$.
    
    The process is called self-similar (of order $H$), if for any $a>0$ and $\ol{\mcl} := (M, a^2 \abrac{.,.}_\mcl)$ the centred Gaussian field $(a^{-H} \ol{X}(\eta))_{\eta\in\ol{\mcl}}$ determined by $\expec{(\ol{X}(\eta) - \ol{X}(\zeta))^2} = g(d_{\ol{\mcl}}(\eta,\zeta)) \\ = g(a ~d_\mcl(\eta,\zeta))$ exists and has the same finite dimensional distributions as $(X(\eta))_{\eta\in\mcl}$. 
\end{defi}

We can immediately apply this concept to our fractional L\'evy fields in the next lemma, the proof of which is immediate.

\begin{lem}\label{lem_selfSimRiemann}
    Let $(X_H(\eta))_{\eta\in\mcl}$ be a fractional L\'evy field. Then it is self-similar of order $H$.
\end{lem}

\begin{rem}
    There exist other Gaussian fields that are self-similar in the sense of Definition~\ref{def_selfSimilarRiemann}, for example seen in \cite[Prop.\ 3.18]{Gelbaum2014}. Note that in \cite[Sec.\ 3.3 and Sec.\ 5]{Istas2012} different notions of self-similarity are used for generalisation.
\end{rem}

We now apply the general self-similarity to the models spaces $\mcl_\kappa$ for any $\kappa\in\R$ by first noting that scaling the Riemannian metric is the same as scaling the radius of the respective model space. Again the proof is a straightforward case distinction ($\kappa<0$, $\kappa=0$, $\kappa>0$) and and explicit computation with the isometry $\vphi:\ol{\mcl_\kappa} \to \mcl_{\kappa/a^2}$ be defined by $\vphi(\eta) := a \eta$.

\begin{lem}\label{lem_ambientScalingIsometriesConstCurvature}
    Let $a>0$ and let $\ol{\mcl_\kappa} = (M_\kappa, a^2 \abrac{.,.}_{\mcl_\kappa})$ be the scaled model space $\mcl_\kappa$. Then $\ol{\mcl_\kappa}$ is isometric to $\mcl_{\kappa/a^2} = (M_{\kappa/a^2}, \abrac{.,.}_{\mcl_{\kappa/a^2}})$.
\end{lem}

We now apply this isometry to the manifold on which the Riemannian metric has been scaled through the self-similarity property of the L\'evy field. 

Using this, we can easily check that a fractional L\'evy field exists on any of the model spaces. The proof is again a straightforward computation.

\begin{lem}\label{lem_fddEqualityConstCurvature}
    Let $(X_H(\eta))_{\eta\in\mcl_\kappa}$ be the fractional L\'evy field on the model space $\mcl_\kappa$ with $\kappa\in\R$. Let $a>0$, let $\ol{\mcl_\kappa}:= (M_\kappa, a^2 \abrac{.,.}_{\mcl_\kappa})$ and let $\vphi: \ol{\mcl_\kappa} \to \mcl_{\kappa/a^2}$ denote the isometry given by $\vphi(\eta):=a \eta$ through Lemma~\ref{lem_ambientScalingIsometriesConstCurvature}. Then a fractional L\'evy field $(\wh{X}_H(\eta))_{\eta\in\mcl_{\kappa/a^2}}$ exists so that $(X_H(\eta))_{\eta\in\mcl_\kappa}$ and $(a^{-H} \wh{X}_H(\vphi(\eta)))_{\eta\in\mcl_{\kappa}}$ possess the same finite dimensional distributions.
\end{lem}

Thus we obtain the following general statement for applying self-similarity of fractional L\'evy fields on model spaces. Note how this procedure changes the underlying space except for the Euclidean case $\kappa=0$.

\begin{prop}\label{prop_selfSimRiemann}
    Let $\kappa\in\R$, let $(X_H(\eta))_{\eta\in\mcl_\kappa}$ be the fractional L\'evy field on the model space $\mcl_\kappa$ and let $0<r< \inj_{\mcl_\kappa}(O)$. For any $a>0$ we have
    \begin{align*}
        \prob{\sup_{\eta\in\bb^{\mcl_\kappa}_r(O)} X_H(\eta) < \eps} = \prob{\sup_{\eta\in\bb^{\mcl_{\kappa/a^2}}_{a r}(O)} \wh{X}_H(\eta) < \eps ~ a^{H}},
    \end{align*}
    where $(\wh{X}_H(\eta))_{\eta\in\mcl_{\kappa/a^2}}$ is the fractional L\'evy field indexed by $\mcl_{\kappa/a^2}$.
\end{prop}
\begin{proof}
    Set $\ol{\mcl_\kappa}=(\mcl_\kappa, a^2 \abrac{.,.}_{\mcl_\kappa})$. The definition of the length of a curve \eqref{eq_curveLengthDef} implies that $d_{\ol{\mcl_\kappa}}(.,.) = a ~ d_{\mcl_\kappa}(.,.)$, which lets us obtain that
    \begin{align}\label{eq_ballScaled}
        \bb^{\mcl_\kappa}_r(O) = \bb_{a r}^{\ol{\mcl_\kappa}}(O).
    \end{align}
    We apply Lemma~\ref{lem_fddEqualityConstCurvature} using the isometry $\vphi: (\mcl_\kappa, a^2 \abrac{.,.}_{\mcl_\kappa}) \to \mcl_{\kappa/a^2}$ and then the ball equality \eqref{eq_ballScaled} to obtain that
    \begin{align*}
        \prob{\sup_{\eta\in\bb^{\mcl_\kappa}_r(O)} X_H(\eta) < \eps} = \prob{\sup_{\eta\in\bb^{\mcl_\kappa}_{r}(O)} a^{-H} \wh{X}_H(\vphi(\eta)) < \eps} = \prob{\sup_{\eta\in\bb^{\ol{\mcl_\kappa}}_{a r}(O)} a^{-H} \wh{X}_H(\vphi(\eta)) < \eps}.
    \end{align*}
    Since $\vphi$ is an isometry we know that $\vphi\left(\bb^{\ol{\mcl_{\kappa}}}_{a r}(O)\right) = \bb^{\mcl_{\kappa/a^2}}_{a r}(O)$ and thus
    \begin{gather*}
        \prob{\sup_{\eta\in\bb^{\ol{\mcl_\kappa}}_{a r}(O)} a^{-H} \wh{X}_H(\vphi(\eta)) < \eps}
        = \prob{\sup_{\eta\in\bb^{\mcl_{\kappa/a^2}}_{a r}(O)} \wh{X}_H(\eta) < \eps ~ a^H}. \qedhere
    \end{gather*}
\end{proof}

\begin{rem}\label{rem_hyperbolicDualProblem}
    Unfortunately, our technique is not applicable to calculate the persistence probability of the hyperbolic fractional L\'evy field on an expanding domain, e.g.\ $\bb^{\hyp_d}_T(O)$, $T\to\infty$, with a fixed barrier. The self-similarity property from Proposition~\ref{prop_selfSimRiemann} implies that
    \begin{align*}
        \prob{\sup_{\eta\in\bb^{\hyp_d}_T(O)} X_H(\eta) < 1} = \prob{\sup_{\eta\in\bb^{\hyp_d(1/T)}_{1}(O)} \wh{X}_H(\eta) < T^{-H}},
    \end{align*}
    which, at first, looks consistent with what we would expect. However, the volume of the $1$-ball in $\mcl_{-T^2} = \hyp_d(1/T)$ increases exponentially in $T$ 
    rather than polynomially and we have heavily relied on the uniform polynomial decay of ball volumes with respect to  their radius (cf.\ Lemma~\ref{lem_ballVolumeBoundsGeneralRiemann}).
\end{rem}

We can, however, calculate the persistence exponent with respect to  fixed domains of arbitrary Riemannian manifolds. For the upper bound we need the result for the persistence exponent of spherical fractional L\'evy fields, cf.\ \eqref{eq_sphericalPersistenceExponent}.
We then prove the general upper bound by comparing the persistence event on the general Riemannian manifold to the spherical case.

\begin{lem}\label{lem_riemannPersistenceUpperBound}
    The persistence probability of the fractional L\'evy field $(X_H(\eta))_{\eta\in K}$ on the relatively compact set $K$ of a connected Riemannian manifold $\mcl$, where $O$ is in the interior of $K$, satisfies
    \begin{align*}
        \prob{\sup_{\eta\in K} X_H(\eta) < \eps} \leq \eps^{\frac{d}{H} + \mathdutchcal{o}(1)}, && \text{ for } \eps \searrow 0.
    \end{align*}
\end{lem}
\begin{proof}
    Since $O$ is in the interior of $K$, we may choose $r'>0$ and $\kcl > 0$ with $\kcl \geq \kcl'$ according to Lemma~\ref{lem_rauchToponogovComparisonEstimate} so that $B_{r'}(O)\subseteq K$, and some $0<r\leq r'$, such that $r = \frac{\pi}{\sqrt{\kcl}}$. This choice implies that $\mcl_\kcl = \sph_d(1/\sqrt{\kcl}) = \sph_d(r/\pi)$ (cf.\ \eqref{eq_modelSpace_injRadius}). Thus, $\inj_{\sph_d(r/\pi)}(O) = r$, i.e.\ $\bb_r^{\sph_d(r/\pi)}(O) = \sph_d(r/\pi)\setminus\{\ol{O}\}$, where $\ol{O}$ is the point antipodal to $O\in\sph_d(r)$.

    Then we may directly apply \eqref{eq_expMapDistancePreservingGeneralised} and Lemma~\ref{lem_rauchToponogovComparisonEstimate} with $\Psi:=\Psi_{\kcl}$ to obtain that
    \begin{align*}
        d_\mcl(O,\eta) &= d_{\sph_d(r/\pi)}(O,\Psi(\eta)),\qquad\qquad
        d_\mcl(\eta,\zeta) \geq d_{\sph_d(r/\pi)}(\Psi(\eta),\Psi(\zeta))
    \end{align*}
    for all $\eta, \zeta\in \bb_r^\mcl(O)$. Thus, if $(\wh{X}_H(\eta))_{\eta\in\sph_d(r/\pi)}$ is a fractional L\'evy field on $\sph_d(r/\pi)$, we obtain that
    \begin{align*}
        \expec{X_H(\eta)^2} & = \expec{\wh{X}_H(\Psi(\eta))^2},\qquad\qquad  
        \expec{X_H(\eta) X_H(\zeta)} & \leq \expec{\wh{X}_H(\Psi(\eta)) \wh{X}_H(\Psi(\zeta))}
    \end{align*}
    for all $\eta, \zeta\in \bb_r^\mcl(O)$ in the same way that we derived the inequality in the proof of Lemma~\ref{lem_hfbmUpperBound}.
    
    By restriction to a smaller domain and then by Slepian's Lemma in the form of Lemma~\ref{lem_Slepian} comparing to the fractional L\'evy field $(\wh{X}_H(\eta))_{\eta\in\sph_{d}(r/\pi)}$ we obtain 
    \begin{align*}
        \prob{\sup_{\eta\in K} X_H(\eta) < \eps}
        \leq \prob{\sup_{\eta\in \bb_r^\mcl(O)} X_H(\eta) < \eps}
        \leq \prob{\sup_{\eta\in \sph_{d}(r/\pi)} \wh{X}_H(\eta) < \eps}.
    \end{align*}
    Recall that the map $\Psi:\bb_r^\mcl(O)\to\bb_r^{\sph_d(r/\pi)}(O) = \sph_d(r/\pi)\setminus\{\ol{O}\}$ is bijective only when the point $\ol{O}$ is excluded. The missing point $\ol{O}$ can be recovered by a.s.\ continuity of the process on the sphere of radius $r/\pi$ and by intersecting with the event of probability one that $\{\wh{X}_H(\ol{O})\neq \eps\}$. We compare to the whole sphere $\sph_d(r/\pi)$, because \eqref{eq_sphericalPersistenceExponent} only works for whole spheres, and not for spherical caps. Using Proposition~\ref{prop_selfSimRiemann} together with the definition of our model spaces \eqref{eq_defModelSpace} and by disregarding antipodal points as before we obtain that
    \begin{align*}
        \prob{\sup_{\eta\in \sph_{d}(r/\pi)} \wh{X}_H(\eta) < \eps}
        = \prob{\sup_{\eta\in \sph_{d}} S_H(\eta) < \eps \left(\frac{\pi}{r}\right)^{H}},
    \end{align*}
    where $(S_H(\eta))_{\eta\in\sph_d}$ is the spherical fractional L\'evy field. The statement now follows from the persistence bound in \eqref{eq_sphericalPersistenceExponent} using that $r$ and $\pi$ are constants.
\end{proof}

The lower bound is obtained similarly by comparing to the standard hyperbolic case.

\begin{lem}\label{lem_riemannPersistenceLowerBound}
    The persistence probability of a fractional L\'evy field $(X_H(\eta))_{\eta\in K}$ on the relatively compact set $K$ of a connected $d$-dimensional Riemannian manifold $\mcl$, where $O$ is in the interior of $K$, satisfies
    \begin{align*}
        \prob{\sup_{\eta\in K} X_H(\eta) < \eps} \geq c ~ \eps^{\frac{d}{H}} ~ \left(\sqrt{-\log \eps}\right)^{-\frac{d}{H}}
    \end{align*}
    for some constant $c>0$ and all $\eps>0$ small enough.
\end{lem}
\begin{proof}
    Since $O$ is in the interior of $K$, we may choose $r>0$ and some $\kappa < 0$ with $\kappa \leq \kappa'$ according to Lemma~\ref{lem_rauchToponogovComparisonEstimate}.
    Now, we split the persistence probability into two pieces
    \begin{align}\label{eq_lowerGeneralRiemannianBoundSplitting}
        \prob{\sup_{\eta\in K} X_H(\eta) < \eps} 
        \geq \prob{\sup_{\eta\in K\setminus \bb_r^\mcl(O)} X_H(\eta) < \eps} ~ \prob{\sup_{\eta\in \bb_r^\mcl(O)} X_H(\eta) < \eps}
    \end{align}
    using Corollary~\ref{cor_SlepianPosCorr}, which is possible, since the field is positively correlated everywhere, cf.\ Remark~\ref{rem_proofOfDomainsWithoutZeroAndPositiveCorrelations}. 
    
    An application of Lemma~\ref{lem_NonPersistenceOnDomainsWithoutZeroRiemann} yields that the probability of the process being non-positive on $K\setminus \bb_r^\mcl(O)$, i.e.\ the first factor in \eqref{eq_lowerGeneralRiemannianBoundSplitting}, is bounded from below by a positive constant.

    Then we may directly apply \eqref{eq_expMapDistancePreservingGeneralised} and Lemma~\ref{lem_rauchToponogovComparisonEstimate} with $\Psi:=\Psi_{\kappa}$ to obtain that
    \begin{align*}
        d_\mcl(O,\eta) &= d_{\mcl_\kappa}(O,\Psi(\eta)), \qquad 
        d_\mcl(\eta, \zeta)  \leq d_{\mcl_\kappa}(\Psi(\eta), \Psi(\zeta))
    \end{align*}

    holds for any $\eta,\zeta\in \bb_r^\mcl(O)$. Therefore, if $(\wh{X}_H(\eta))_{\eta\in\mcl_\kappa}$ is the fractional L\'evy field on $\mcl_\kappa$, we obtain that
    \begin{align*}
        \expec{X_H(\eta)^2} & = \expec{\wh{X}_H(\Psi(\eta))^2}, \qquad 
        \expec{X_H(\eta) X_H(\zeta)}  \geq \expec{\wh{X}_H(\Psi(\eta)) \wh{X}_H(\Psi(\zeta))}
    \end{align*}
    for any $\eta,\zeta\in\bb_r^\mcl(O)$ in the same way that we derived the reverse inequality in the proof of Lemma~\ref{lem_hfbmUpperBound}. 
    From the above estimate \eqref{eq_lowerGeneralRiemannianBoundSplitting}, Lemma~\ref{lem_NonPersistenceOnDomainsWithoutZeroRiemann},  Slepian's Lemma (Lemma~\ref{lem_Slepian}) and the self-similarity property for persistence probabilities in model spaces from Proposition~\ref{prop_selfSimRiemann} with $a:=\sqrt{\abs{\kappa}}$ we may thus infer the lower bound
    \begin{align*}
        \prob{\sup_{\eta\in K} X_H(\eta) < \eps} 
        &\geq c' ~ \prob{\sup_{\eta\in \bb_r^{\mcl}(O)} X_H(\eta) < \eps}
        \geq c' ~ \prob{\sup_{\eta\in \bb_r^{\mcl_\kappa}(O)} \wh{X}_H(\eta) < \eps}\\
        &= c' ~ \prob{\sup_{\eta\in \bb_{r \sqrt{\abs{\kappa}}}^{\hyp_d}(O)} X_H^{\hyp_d}(\eta) < \eps ~ \abs{\kappa}^{H/2}}
        \geq c ~ \eps^{\frac{d}{H}} ~ \left(\sqrt{-\log \eps}\right)^{-\frac{d}{H}},
    \end{align*}
    where in the last step we used Lemma~\ref{lem_hfbmLowerBound} for $R:=r \sqrt{\abs{\kappa}}$ and $\eps>0$ small enough and where the constants $c,c'>0$ may depend on $\mcl, \kappa, H, r$. The stochastic process $(X_H^{\hyp_d}(\eta))_{\eta\in\hyp_d}$ is the hyperbolic fractional L\'evy field. This proves the statement.
\end{proof}

Now Lemma~\ref{lem_riemannPersistenceUpperBound} and Lemma~\ref{lem_riemannPersistenceLowerBound} imply Theorem~\ref{thm_mainRiemannianManifold}.

\bibliographystyle{alpha}
\bibliography{bibliography.bib}

\begin{thebibliography}{BdlHV08}

\bibitem[AB76]{AskeyBingham1976}
R.~Askey and N.~H. Bingham.
\newblock {Gaussian processes on compact symmetric spaces}.
\newblock {\em Zeitschrift für Wahrscheinlichkeitstheorie und verwandte
  Gebiete}, 37(2):127–143, 1976.

\bibitem[AE09]{AmannEscher2009book}
H.~Amann and J.~Escher.
\newblock {\em Analysis III}.
\newblock Birkh\"{a}user Basel, 2009.

\bibitem[AH26]{AurzadaHelmer2026}
F.~Aurzada and M.~Helmer.
\newblock Persistence probabilities of spherical fractional brownian motion.
\newblock {\em ALEA Lat. Am. J. Probab. Math. Stat.}, 23(1):221, 2026.

\bibitem[Arc92]{Arcones1992}
M.~A. Arcones.
\newblock On the arg max of a {G}aussian process.
\newblock {\em Statistics \& Probability Letters}, 15(5):373--374, 1992.

\bibitem[Arc95]{Arcones1995}
M.~A. Arcones.
\newblock On the law of the iterated logarithm for gaussian processes.
\newblock {\em Journal of Theoretical Probability}, 8(4):877–903, 1995.

\bibitem[AS15]{AurzadaSimon2015}
F.~Aurzada and T.~Simon.
\newblock {Persistence probabilities and exponents}.
\newblock In {\em L\'evy matters. {V}}, volume 2149 of {\em Lecture Notes in
  Mathematics}, pages 183--224. Springer, Cham, 2015.

\bibitem[AT09]{AdlerTaylor2009}
R.~J. Adler and J.~E. Taylor.
\newblock {\em {Random Fields and Geometry}}.
\newblock Springer Monographs in Mathematics. Springer New York, 2009.

\bibitem[Bau13]{Baumgarten2013thesis}
C.~Baumgarten.
\newblock {\em {Persistence of sums of independent random variables, iterated
  processes and fractional Brownian motion}}.
\newblock PhD thesis, Technische Universit\"at Berlin, 2013.
\newblock \url{https://depositonce.tu-berlin.de/handle/11303/3917}.

\bibitem[BdlHV08]{BekkaEtAl2008book}
B.~Bekka, P.~de~la Harpe, and A.~Valette.
\newblock {\em {Kazhdan’s Property (T)}}.
\newblock New Mathematical Monographs. Cambridge University Press, 2008.

\bibitem[BHOZ08]{BiaginiEtAl2008book}
F.~Biagini, Y.~Hu, B.~\O{}ksendal, and T.~Zhang.
\newblock {\em {Stochastic calculus for fractional {B}rownian motion and
  applications}}.
\newblock Probability and its Applications (New York). Springer-Verlag London,
  Ltd., London, 2008.

\bibitem[BMS13]{BrayMajumdarSchehr2013}
A.~J. Bray, S.~N. Majumdar, and G.~Schehr.
\newblock {Persistence and first-passage properties in nonequilibrium systems}.
\newblock {\em Advances in Physics}, 62:225--361, 2013.

\bibitem[Cha01]{Chavel2001}
I.~Chavel.
\newblock {\em Isoperimetric inequalities}, volume 145 of {\em Cambridge Tracts
  in Mathematics}.
\newblock Cambridge University Press, Cambridge, 2001.

\bibitem[Cha06]{Chavel2006}
I.~Chavel.
\newblock {\em Riemannian Geometry: A Modern Introduction}.
\newblock Cambridge University Press, 2006.

\bibitem[Che57]{Chentsov1957}
N.~N. Chentsov.
\newblock Lévy {B}rownian motion for several parameters and generalized white
  noise.
\newblock {\em Theory of Probability {\&} Its Applications}, 2(2):265–266,
  1957.

\bibitem[Chr70]{Christensen1970}
J.~P.~R. Christensen.
\newblock On some measures analogous to {H}aar measure.
\newblock {\em Mathematica Scandinavica}, 26:103--106, 1970.

\bibitem[CL12]{CohenLifshits2012}
S.~Cohen and M.~A. Lifshits.
\newblock Stationary {G}aussian random fields on hyperbolic spaces and on
  {E}uclidean spheres.
\newblock {\em ESAIM: Probability and Statistics}, 16:165–221, 2012.

\bibitem[Far73]{Faraut1973}
J.~Faraut.
\newblock Fonction brownienne sur une vari\'et\'e riemannienne.
\newblock {\em S\'eminaire de probabilit\'es}, 7:61--76, 1973.

\bibitem[FH74]{Faraut1974}
J.~Faraut and K.~Harzallah.
\newblock Distances hilbertiennes invariantes sur un espace homogène.
\newblock {\em Annales de l’Institut Fourier}, 24(3):171–217, 1974.

\bibitem[FLH15]{FeragenLauzeHauberg2015}
A.~Feragen, F.~Lauze, and S.~Hauberg.
\newblock Geodesic exponential kernels: When curvature and linearity conflict.
\newblock In {\em 2015 IEEE Conference on Computer Vision and Pattern
  Recognition (CVPR)}, pages 3032--3042, 2015.

\bibitem[Gan67]{Gangolli1967}
R.~Gangolli.
\newblock Positive definite kernels on homogeneous spaces and certain
  stochastic processes related to {L}évy's brownian motion of several
  parameters.
\newblock {\em Annales de l’Institut Henri Poincaré Probabilités et
  statistiques}, 3(2):121--226, 1967.

\bibitem[Gel14]{Gelbaum2014}
Z.~A. Gelbaum.
\newblock {Fractional Brownian fields over manifolds}.
\newblock {\em Transactions of the American Mathematical Society},
  366(9):4781--4814, 2014.

\bibitem[Her60]{Hermann1960Remarks}
R.~Hermann.
\newblock Remarks on the foundations of integral geometry.
\newblock {\em Rendiconti del Circolo Matematico di Palermo}, 9(1):91–96,
  1960.

\bibitem[HKM02]{HjorthKokkendorffMarkvorsen2002}
P.~G. Hjorth, S.~L. Kokkendorff, and S.~Markvorsen.
\newblock Hyperbolic spaces are of strictly negative type.
\newblock {\em Proceedings of the American Mathematical Society},
  130(1):175--181, 2002.

\bibitem[Ist05]{Istas2005}
J.~Istas.
\newblock {Spherical and hyperbolic hractional Brownian motion}.
\newblock {\em Electronic Communications in Probability}, 10:254 -- 262, 2005.

\bibitem[Ist06]{Istas2006stable}
J.~Istas.
\newblock On fractional stable fields indexed by metric spaces.
\newblock {\em Electronic Communications in Probability}, 11, 2006.

\bibitem[Ist12]{Istas2012}
J.~Istas.
\newblock {Manifold indexed fractional fields}.
\newblock {\em ESAIM: Probability and Statistics}, 16:222--276, 2012.

\bibitem[Kar89]{Karcher1989}
H.~Karcher.
\newblock Riemannian comparison constructions.
\newblock In {\em Global differential geometry}, volume~27 of {\em MAA Stud.
  Math.}, pages 170--222. Math. Assoc. America, Washington, DC, 1989.

\bibitem[Kol40]{Kolmogoroff1940}
A.~N. Kolmogoroff.
\newblock Wienersche {S}piralen und einige andere interessante {K}urven im
  {H}ilbertschen {R}aum.
\newblock {\em C. R. (Doklady) Acad. Sci. URSS (N.S.)}, 26:115--118, 1940.

\bibitem[KP90]{KimPollard1990}
J.~K. Kim and D.~Pollard.
\newblock Cube root asymptotics.
\newblock {\em The Annals of Statistics}, 18(1):191--219, 1990.

\bibitem[KU22]{KraetschmerUrusov2022}
V.~Kr\"{a}tschmer and M.~Urusov.
\newblock {A {K}olmogorov–{C}hentsov {T}ype {T}heorem on {G}eneral {M}etric
  {S}paces with {A}pplications to {L}imit {T}heorems for {B}anach-{V}alued
  {P}rocesses}.
\newblock {\em Journal of Theoretical Probability}, 36(3):1454–1486, 2022.

\bibitem[Lee12]{Lee2012book}
J.~M. Lee.
\newblock {\em Introduction to Smooth Manifolds}.
\newblock Springer New York, 2012.

\bibitem[L{\'{e}}v40]{Levy1940}
P.~L{\'{e}}vy.
\newblock {Le Mouvement Brownien Plan}.
\newblock {\em Amer. J. Math.}, 62:487--550, 1940.

\bibitem[L{\'{e}}v65]{Levy1965book}
P.~L{\'{e}}vy.
\newblock {\em {Processus stochastiques et mouvement brownien}}.
\newblock Gauthier-Villars \& Cie, Paris, 1965.

\bibitem[LP18]{LopezPimentel2018}
S.~I. López and L.~P.~R. Pimentel.
\newblock On the location of the maximum of a process: Lévy, gaussian and
  random field cases.
\newblock {\em Stochastics}, 90(8):1221--1237, 2018.

\bibitem[LP23]{LucicPasqualetto2023}
D.~Lučić and E.~Pasqualetto.
\newblock The metric-valued {L}ebesgue differentiation theorem in measure
  spaces and its applications.
\newblock {\em Advances in Operator Theory}, 8(2), 2023.

\bibitem[LSSW16]{LodhiaEtAl2016}
A.~Lodhia, S.~Sheffield, X.~Sun, and S.~S. Watson.
\newblock {Fractional {G}aussian fields: {A} survey}.
\newblock {\em Probability Surveys}, 13:1--56, 2016.

\bibitem[MC68]{ChentsovMorozova1968}
E.~A. Morozova and N.~N. Chentsov.
\newblock {P. Lévy’s Random Fields}.
\newblock {\em Theory of Probability \& Its Applications}, 13(1):153--156,
  1968.

\bibitem[Mis08]{Mishura2008book}
Y.~S. Mishura.
\newblock {\em {Stochastic calculus for fractional {B}rownian motion and
  related processes}}, volume 1929 of {\em Lecture Notes in Mathematics}.
\newblock Springer-Verlag, Berlin, 2008.

\bibitem[Mol67]{Molchan1967}
G.~M. Molchan.
\newblock {On some problems concerning Brownian motion in Lévy’s sense}.
\newblock {\em Theory of Probability \& Its Applications}, 12(4):682--690,
  1967.

\bibitem[Mol79]{Molchan1979}
G.~M. Molchan.
\newblock Homogeneous random fields on symmetric spaces of rank one.
\newblock {\em Teor. Veroyatnost. i Mat. Statist.}, 21:123--148, 167, 1979.

\bibitem[Mol88]{Molino1988book}
P.~Molino.
\newblock {\em Riemannian Foliations}.
\newblock Birkh\"{a}user Boston, 1988.

\bibitem[Mol99]{Molchan99}
G.~M. Molchan.
\newblock {Maximum of a fractional {B}rownian motion: probabilities of small
  values}.
\newblock {\em Communications in Mathematical Physics}, 205(1):97--111, 1999.

\bibitem[Mol17]{Molchan2017}
G.~M. Molchan.
\newblock {Survival exponents for fractional {B}rownian motion with
  multivariate time}.
\newblock {\em ALEA Lat. Am. J. Probab. Math. Stat.}, 14(1):1--7, 2017.

\bibitem[Mol18]{Molchan2018}
G.~M. Molchan.
\newblock {Persistence exponents for {G}aussian random fields of fractional
  {B}rownian motion type}.
\newblock {\em Journal of Statistical Physics}, 173(6):1587--1597, 2018.

\bibitem[MOR14]{MetzlerEtAl2014book}
R.~Metzler, G.~Oshanin, and S.~Redner, editors.
\newblock {\em First-passage phenomena and their applications}.
\newblock World Scientific Publishing Co. Pte. Ltd., Hackensack, NJ, 2014.

\bibitem[MVN68]{MandelbrotVanNess1968}
B.~B. Mandelbrot and J.~W. Van~Ness.
\newblock Fractional {B}rownian motions, fractional noises and applications.
\newblock {\em SIAM Review}, 10:422--437, 1968.

\bibitem[Nou12]{Nourdin2012book}
I.~Nourdin.
\newblock {\em {Selected aspects of fractional {B}rownian motion}}, volume~4 of
  {\em Bocconi \& Springer Series}.
\newblock Springer, Milan; Bocconi University Press, Milan, 2012.

\bibitem[Pet16]{Petersen2016}
P.~Petersen.
\newblock {\em {Riemannian Geometry}}.
\newblock Springer International Publishing, third edition, 2016.

\bibitem[Pey17]{Peyre2017}
R.~Peyre.
\newblock Fractional {B}rownian motion satisfies two-way crossing.
\newblock {\em Bernoulli}, 23(4B):3571--3597, 2017.

\bibitem[Rat19]{Ratcliffe2019}
J.~G. Ratcliffe.
\newblock {\em Foundations of Hyperbolic Manifolds}.
\newblock Springer International Publishing, 2019.

\bibitem[Rit23]{Ritore2023}
M.~Ritoré.
\newblock {\em Isoperimetric Inequalities in Riemannian Manifolds}.
\newblock Springer International Publishing, 2023.

\bibitem[San04]{Santaló2004}
L.~A. Santaló.
\newblock {\em Integral Geometry and Geometric Probability}.
\newblock Cambridge Mathematical Library. Cambridge University Press, 2nd
  edition, 2004.

\bibitem[Sch38]{Schoenberg1938}
I.~J. Schoenberg.
\newblock Metric spaces and positive definite functions.
\newblock {\em Transactions of the American Mathematical Society},
  44(3):522–536, 1938.

\bibitem[She16]{Shen2016}
Y.~Shen.
\newblock Random locations, ordered random sets and stationarity.
\newblock {\em Stochastic Processes and their Applications}, 126(3):906--929,
  2016.

\bibitem[Sle62]{Slepian62}
D.~Slepian.
\newblock {The one-sided barrier problem for {G}aussian noise}.
\newblock {\em The Bell System Technical Journal}, 41(2):463--501, 1962.

\bibitem[SS13]{SamorodnitskyShen2013}
G.~Samorodnitsky and Y.~Shen.
\newblock Is the location of the supremum of a stationary process nearly
  uniformly distributed?
\newblock {\em The Annals of Probability}, 41(5):3494--3517, 2013.

\bibitem[SSV12]{SchillingSongVondracek2012}
R.~L. Schilling, R.~Song, and Z.~Vondracek.
\newblock {\em Bernstein Functions}.
\newblock De Gruyter, Berlin, Boston, 2012.

\bibitem[Tak87]{Takenaka1987}
S.~Takenaka.
\newblock Representation of {E}uclidean {R}andom {F}ield.
\newblock {\em Nagoya Mathematical Journal}, 105:19–31, 1987.

\bibitem[Tak91]{Takenaka1991}
S.~Takenaka.
\newblock Integral-geometric construction of self-similar stable processes.
\newblock {\em Nagoya Mathematical Journal}, 123:1–12, 1991.

\bibitem[TKU81]{TakenakaKuboUrakawa1981}
S.~Takenaka, I.~Kubo, and H.~Urakawa.
\newblock {B}rownian motion parametrized with metric space of constant
  curvature.
\newblock {\em Nagoya Mathematical Journal}, 82:131–140, 1981.

\bibitem[Ven16]{VenetThesis2016}
N.~Venet.
\newblock {\em {On the existence of fractional brownian fields indexed by
  manifolds}}.
\newblock PhD thesis, {Universit{\'e} Paul Sabatier - Toulouse III}, 2016.
\newblock \url{https://theses.hal.science/tel-01825845}.

\bibitem[Vid67]{Vidal1967}
E.~Vidal.
\newblock On regular foliations.
\newblock {\em Annales de l’Institut Fourier}, 17(1):129–133, 1967.

\end{thebibliography}

\appendix
\section{Appendix}

\subsection{Supplementary facts and proofs}

\begin{lem}\label{lem_ballVolume}
    The volume $\sigma(\bb_r^{\mcl_\kappa}(O))$ with respect to the Riemannian measure $\sigma$ of the geodesic ball with radius $0<r<\inj(\mcl_\kappa)$ in the $d$-dimensional model space $\mcl_\kappa$ satisfies:
    \begin{align}\label{eq_ballVolumeBoundsModelSurfaces}
        c~\eps^d \leq \sigma(\bb_\eps^{\mcl_\kappa}(O)) \leq C ~ \eps^{d}
    \end{align}
    for constants $c,C>0$ depending on $\kappa$ and $r$ and all $0<\eps<r$.
\end{lem}

The lemma can be inferred from the explicit formulas for $\sigma(\bb_\eps^{\mcl_\kappa}(O))$ given in  \cite[Eq.\ III.4.1]{Chavel2006} together with \cite[Eq.\ II.5.8]{Chavel2006}.

\begin{proof}[Proof of Lemma~\ref{lem_ballInequalityImpliesSurfaceDensity}]
    We show that for all null sets $D\subseteq A$ of the Riemannian measure $\sigma$ we have that $\prob{X\in D} = 0$. 
    Recall the definition of the Hausdorff-measure of dimension $s>0$ (cf.\ \cite[Sec.\ IX.3]{AmannEscher2009book}, \cite[Sec.\ 1.5.2]{Ritore2023}, \cite[Sec.\ §III.5]{Chavel2006}): For $E\subseteq\mcl$ define
    \begin{align*}
        \hcl_\eps^s (E) := \inf\left\{ \sum_{i=1}^\infty (\diam B_i)^s : E \subseteq \bigcup_{i=1}^\infty B_i, ~\diam(B_i) < \eps \text{ and $B_i$ open for all $i=1,2,\ldots$} \right\},
    \end{align*}
    where the diameter of a set in a metric space is defined in the usual sense as $\diam(A):=\sup_{x,y\in A}{d(x,y)}$.
    The $s$-dimensional Hausdorff measure is then defined by $\hcl^s :=\lim_{\eps\to 0} \hcl_\eps^s$. On $\mcl$ the Hausdorff measure of dimension $d$ is up to a constant identical to the Riemannian measure (cf.\ \cite[Sec.\ 1.5.2]{Ritore2023} or \cite[Sec.\ IV]{Chavel2001}), i.e.\ $\mathcal{H}^d = \alpha_{\mcl} ~ \sigma$.
    
    Since $D$ is a null set for the measure $\sigma$ it is also a null set for the $d$-dimensional Hausdorff measure. Thus, for any $\delta>0$ there exists a cover
    \begin{align*}
        D \subseteq \bigcup_{i=1}^\infty B_i, \quad \text{such that}\quad \sum_{i=1}^\infty (\diam B_i)^d < \frac{\delta}{c_A} \quad \text{with}\quad \diam(B_i) < \eps_A \text{ for all $i=1,\ldots$},
    \end{align*}
    where $c_A>0$ and $\eps_A>0$ are given by the assumption.
    Set $\eps_i:=\diam(B_i)$ and choose arbitrary $\eta_i\in B_i\cap D$ for every $i=1,\ldots$. This is possible, since if there is $j\geq 0$ with $B_j\cap D=\emptyset$ we can just choose the cover not containing $B_j$ without loss of generality. Then it still holds that $D\subseteq \bigcup_{i=1}^\infty \bb_{\eps_i}(\eta_i)$, since the radii were defined so that $B_i\subseteq \bb_{\eps_i}(\eta_i)$ for all $i=1,\ldots$.
    Therefore, using the assumption on $X$ we calculate that
    \begin{align*}
        \prob{X\in D} \leq \prob{X \in \bigcup_{i=1}^\infty \bb_{\eps_i}(\eta_i)}
        \leq \sum_{i=1}^\infty \prob{X\in \bb_{\eps_i}(\eta_i)}
        \leq \sum_{i=1}^\infty c_A ~ \eps_i^d 
        = c_A \sum_{i=1}^\infty (\diam B_i)^d< \delta.
    \end{align*}
    Since $\delta$ was chosen arbitrarily, this shows that $\prob{X\in D} = 0$. 
\end{proof}

\subsection{Local existence of fractional L\'evy fields}

To stay as self-contained as possible we provide the full proof of Lemma~\ref{lem_mainExistence2dFBM} with the exception of one result in Riemannian geometry for which we refer the reader to the literature. We follow the steps of \cite[Sec.\ 2]{VenetThesis2016} to a large degree.

Recall the definition of a Gaussian white noise measure (cf.\ \cite[Sec.\ 1.4.3]{AdlerTaylor2009}, \cite[Def.\ 2.1]{VenetThesis2016}, \cite{ChentsovMorozova1968}).

\begin{defi}\label{def_whiteNoiseMeasure}
    Let $(E, \ecl, \nu)$ be a measure space and let $\mathcal{E}_\nu$ denote all sets in $\mathcal{E}$ with finite $\nu$-measure. A random field $\wcl:\ecl_\nu\to\R$, such that for any $A,B\in\ecl_\nu$ we have
    \begin{enumerate}
        \item $\wcl(A)\sim\ncl(0, \nu(A))$,
        \item $\wcl(A \cup B) = \wcl(A) + \wcl(B) - \wcl(A\cap B)$ almost surely,
        \item $A\cap B = \emptyset$ implies that $\wcl(A)$ and $\wcl(B)$ are independent,
    \end{enumerate}
    is called Gaussian white noise with control measure $\nu$.
\end{defi}

It is well known that such a measure exists for any measure space (cf.\ \cite[Thm.\ 1.4.3]{AdlerTaylor2009}).

\begin{lem}\label{lem_existWhiteNoise}
    For any measure space $(E, \ecl, \nu)$ there exists a Gaussian white noise with control measure $\nu$ so that $\expec{\wcl(A)\wcl(B)}= \nu(A\cap B)$.
\end{lem}

Given a connected $2$-dimensional Riemannian manifold $\mcl$ without border and with a point $O$ in it, our goal is to define a measure space $(E, \ecl, \nu)$ so that there is a neighbourhood $U$ around $O$ and a function $\eta\mapsto A_\eta$ mapping points in $U$ to sets in $\ecl_\nu$ so that $\nu(A_\eta \cap A_\zeta) = d(\eta,\zeta)$. Then the Gaussian white noise measure with control measure $\nu$ is the L\'evy Brownian field on $U$. The existence of the fractional L\'evy field will then follow through a different argument, cf.\ Lemma~\ref{lem_BrownianFieldToFractionallevy} below.

Recall from Section~\ref{sec_RiemannianBackground} that the generalisation of straight lines on Riemannian manifolds are called geodesics and that geodesics are uniquely defined by a starting point and a starting direction.

We introduce the notion of \emph{strongly (geodesically) convex} domains (cf.\ \cite[Sec.\ IX.6]{Chavel2006}):
\begin{defi}
    A domain $U$ on a Riemannian manifold $\mcl$ is called \emph{strongly (geodesically) convex} if for any $\eta,\zeta$ there is a unique geodesic between $\eta$ and $\zeta$ contained in $U$ which is length minimising in $U$ and length minimising in the ambient manifold $\mcl$.
\end{defi}

Working in a strongly convex domain removes many problems that occur in unrestricted Riemannian manifolds. Importantly, such a domain always exists around any point as can be inferred from the fact that the injectivity radius in any point is positive (cf.\ \cite[Thm.\ III.2.3]{Chavel2006}) together with the proof of \cite[Thm.\ IX.6.1]{Chavel2006}.

Now we present the appropriate control measure for the Gaussian white noise that will give us the L\'evy Brownian field. The construction can be found in \cite[Sec.\ IV.19.1, IV.19.2 and IV.19.4]{Santaló2004}. For additional technical details, cf.\ \cite[Sec.\ 1 \& 2]{Hermann1960Remarks}, \cite[Sec.\ 2.3, p.39]{Molino1988book} and \cite{Vidal1967}. Measures of this type are called kinematic measures and the resulting integral is sometimes called a \emph{Crofton formula}. Usually they require strong (symmetry) assumptions on the underlying space, which is not required here and which is why the existence of this measure is remarkable.

\begin{lem}\label{lem_santaloMeasure}
    For any interior point $O\in\mcl$ on a $2$-dimensional Riemannian surface there exists a strongly geodesically convex neighbourhood $U$ around $O$ and a measure $\nu_G$ on the $2$-dimensional set of geodesics $G$ in $U$, such that the following holds: If $\Gamma$ is a curve in $U$ with length $L(\Gamma)$ and $N_\Gamma(\gamma)$ denotes the number of intersections of $\Gamma$ with $\gamma\in G$ then we have
    \begin{align*}
        \int_G N_\Gamma(\gamma) \dd \nu_G(\gamma) = L(\Gamma).
    \end{align*}
\end{lem}

In strongly convex domains $U$, for any two points $\eta,\zeta\in U$ there is exactly one geodesic $\gamma_{\eta\zeta}$ entirely contained within $U$ passing through these points. This geodesic is length minimising in $U$ and in the ambient manifold. We write $\eta\zeta$ as shorthand notation for the curve defined by this unique geodesic between $\eta$ and $\zeta$ so that the length of the curve is precisely $d(\eta,\zeta)$. Furthermore, we define $\ol{\eta\zeta}$ to be the set of geodesics in $U$ intersecting $\eta\zeta$, but excluding $\gamma_{\eta\zeta}$. The mapping $\eta\mapsto\ol{O\eta}$ is the connection between points on the manifold and sets that we can measure with the control measure. We obtain the following corollary.

\begin{cor}\label{cor_santaloMeasureSpecific}
    Let the setup be the same as in Lemma~\ref{lem_santaloMeasure}. If $\Gamma = \eta\zeta$ for some $\eta,\zeta\in U$ then
    \begin{align*}
        \nu_G(\ol{\eta\zeta}) = d(\eta,\zeta)
    \end{align*}
\end{cor}
\begin{proof}
    Let $\gamma_{\eta\zeta}$ be the unique geodesic passing through $\eta$ and $\zeta$ in $U$. Then $\nu_G(\gamma_{\eta\zeta})$ must be equal to zero, since otherwise Lemma~\ref{lem_santaloMeasure} would imply that any curve passing through $\eta$ had length of at least $\nu_G(\gamma_{\eta\zeta})$, which is nonsense. Due to strong geodesic convexity, the number of intersections that $\Gamma$ has with any $\gamma\in G\setminus\{\gamma_{\eta\zeta}\}$ is at most one. Thus, we may replace $N_\Gamma(\gamma)$ by an indicator function and obtain that
    \begin{align*}
        \nu_G(\ol{\eta\zeta}) + 0 = \int_{G\setminus\{\gamma_{\eta\zeta}\}} \ind_{\Gamma\cap\gamma\neq\emptyset} \dd \nu_G(\gamma) + \nu_G(\gamma_{\eta\zeta}) = \int_{G} N_\Gamma(\gamma) \dd \nu_G(\gamma) = L(\Gamma) = d(\eta,\zeta),
    \end{align*}
    since $\Gamma$ is length minimising by strong geodesic convexity and since $\int_{\gamma_{\eta\zeta}} N_\Gamma(\gamma)\dd \nu_G(\gamma) = 0$ despite $N_\Gamma(\gamma_{\eta\zeta})=\infty$, because $\nu_G(\gamma_{\eta\zeta})=0$.
\end{proof}

A geodesic triangle $(\eta,\zeta,\xi)$ in a strongly convex domain $U$ consists of the three geodesic edges $\eta\zeta, \zeta\xi, \xi\eta$ connecting the given three distinct vertex points. The following lemma shows that the intersection of triangles and geodesics behaves the same way as in the Euclidean case. The statement in Euclidean space is sometimes referred to as ``Pasch's axiom''.

\begin{lem}\label{lem_paschsAxiom}
    In a strongly convex domain $U$ every geodesic line that intersects a geodesic triangle $(\eta, \zeta, \xi)$ in a non-vertex point intersects the triangle exactly twice.
\end{lem}
\begin{proof}
    A geodesic is completely determined by its position and its direction. Let $\gamma$ be a geodesic that intersects an edge of the triangle at least once with no intersection in any of the vertex points.

    The intersection cannot be tangential, since this would imply that the geodesic follows the same direction in the same point as that edge of the triangle. Thus, $\gamma$ would contain that edge and therefore also one of the vertex points, which is a contradiction.

    The triangle is compact and defines an interior and an exterior by diffeomorphism to the Euclidean plane.
    
    Suppose that $\gamma$ and the triangle intersected only at the single non-vertex point $\xi$: Since the intersection is not tangential there is a part of $\gamma$ inside the interior of the triangle. This would imply that we could extend $\gamma$ infinitely in the interior of the triangle without ever hitting it. However, $\gamma$ cannot be of arbitrary length: The triangle together with its interior is a compact domain and the distance $d(\xi,.)$ is a continuous function. Thus, this distance is bounded on compact domains. Since $\gamma$ travels with a constant non-zero velocity vector by definition, cf.\ Section~\ref{sec_RiemannianBackground}, its length scales linearly with the time travelled, cf.\ \eqref{eq_curveLengthDef}. However, its length at time $t$ is exactly the distance $d(\xi,\gamma(t))$ because of the strong convexity of the domain. Thus, if $\gamma$ were to travel infinitely long, then the distance of $\xi$ to a point inside the triangle would be unbounded, which cannot be as we have already established by compactness. Therefore, it must hit another edge of the triangle in finite time.

    Suppose now that $\gamma$ and the triangle intersected in more than two points: By repeating the same arguments as before, we know that the intersection must occur an even number of times. Now, if there were to be more than exactly two points of intersection then two of those intersections would have to lie on the exact same edge of the triangle. However, this would imply that $\gamma$ is the shortest path between two points on the same side of the triangle. Thus, $\gamma$ would be identical to the geodesic that defines this particular side of the triangle, which cannot be true, since we assumed that $\gamma$ does not intersect any of the vertices.

    Therefore, every geodesic $\gamma$ which spends time in the interior of the triangle, must intersect the triangle exactly twice.
\end{proof}

The previous lemma automatically implies the following, which is derived from \cite[Lem.\ 2.2]{VenetThesis2016}. Note that for sets $A,B$ the \textit{symmetric difference} $\Delta$ is defined as $A\Delta B := A\setminus B \cup B\setminus A$.

\begin{cor}\label{cor_symmetricDiff}
    For any three distinct points $\eta,\zeta,\xi\in U$, where $U$ is a strongly geodesically convex neighbourhood, we have that
    \begin{align*}
        \ol{\eta\zeta} = \ol{\eta\xi} ~\Delta ~\ol{\zeta\xi}.
    \end{align*}
\end{cor}

Note that the identity is correct, since for any $\eta,\zeta\in U$ we excluded the unique geodesic $\gamma_{\eta\zeta}$ passing through $\eta$ and $\zeta$.

We have now collected all the ingredients necessary for the existence proof of the local existence of the L\'evy Brownian field. This is analogous to \cite[Thm.\ 2.1]{VenetThesis2016}.

\begin{lem}\label{lem_existenceLévyBrownianField2D}
    Let $\mcl$ be a $2$-dimensional connected Riemannian manifold without boundary and with $O\in\mcl$ in its interior. Then there exists a neighbourhood $U$ around $O$, on which the L\'evy Brownian field exists.
\end{lem}
\begin{proof}
    We choose $U$ as the strongly geodesically convex domain, so that the measure $\nu_G$ from Lemma~\ref{lem_santaloMeasure} exists. Then by Lemma~\ref{lem_existWhiteNoise} there is a Gaussian white noise $\wcl$ with control measure $\nu_G$, such that we can define
    \begin{align*}
        X_\eta := \wcl(\ol{O\eta})
    \end{align*}
    for any $\eta\in U$. Then $(X_\eta)_{\eta\in U}$ is a centred Gaussian process. Note that
    \begin{align}\label{eq_nuGSymmetricDifference}
        \nu_G(\ol{O\eta}) + \nu_G(\ol{O\zeta}) - \nu_G(\ol{O\eta}~ \Delta ~\ol{O\zeta}) = 2\nu_G(\ol{O\eta}\cap\ol{O\zeta}).
    \end{align}
    
    Thus, we can calculate the covariance of $(X_\eta)_{\eta\in U}$ for $\eta,\zeta\in U$ using Lemma~\ref{lem_existWhiteNoise}, \eqref{eq_nuGSymmetricDifference} and Corollary~\ref{cor_symmetricDiff} through
    \begin{align*}
        \expec{X_\eta X_\zeta} 
        & = \expec{\wcl(\ol{O\eta}) ~ \wcl(\ol{O\zeta})}
        = \nu_G(\ol{O\eta} \cap \ol{O\zeta}) \\
        & = \frac{1}{2}\left( \nu_G(\ol{O\eta}) + \nu_G(\ol{O\zeta}) - \nu_G(\ol{O\eta}~\Delta ~\ol{O\zeta}) \right) \\
        & = \frac{1}{2}\left( \nu_G(\ol{O\eta}) + \nu_G(\ol{O\zeta}) - \nu_G(\ol{\eta\zeta}) \right)\\
        & = \frac{1}{2}\left( d(O,\eta) + d(O,\zeta) - d(\eta,\zeta) \right),
    \end{align*}
    where in the last step we applied Corollary~\ref{cor_santaloMeasureSpecific}. A simple rearrangement shows that this covariance defines the process with structure function
    \begin{gather*}
        \expec{(X_\eta-X_\zeta)^2} = d(\eta,\zeta). \qedhere
    \end{gather*}
\end{proof}

It remains to show the existence of the \emph{fractional} L\'evy field for the other Hurst parameters $0<H<1/2$. The following definition (cf.\ \cite[Sec.\ 2.4]{Istas2012}) in combination with the lemma following thereafter will allow for simpler notation.

\begin{defi}\label{def_condNegSemiDef}
    Let $A$ be a set and let $f:A\times A\to \R_{\geq 0}$ be a symmetric function. Then it is called conditionally negative definite, if for any $n\in\N$, any $\lambda_1, \ldots, \lambda_n \in\R$ with $\lambda_1 + \ldots + \lambda_n=0$ and any $x_1, \ldots, x_n \in A$ it holds that
    \begin{align*}
        \sum_{i,j=1}^n \lambda_i \lambda_j f(x_i, x_j) \leq 0.
    \end{align*}
\end{defi}

Now we may state the following version of Schoenberg's theorem (cf.\ \cite[Sec.\ 2.4]{Istas2012}, \cite[Sec.\ C.3]{BekkaEtAl2008book}, \cite{Schoenberg1938}).

\begin{lem}\label{lem_schoenbergsTheorem}
    Let $A$ be a set and let $f:A\times A\to\R_{\geq 0}$ be a symmetric function. Let $O\in A$ be an arbitrary point and let $g:A\times A\to \R_{\geq 0}$ be defined by
    \begin{align*}
        g(x,y)&:= f(O,x) + f(O,y) - f(x,y).
    \end{align*} 
    \begin{enumerate}
        \item $f$ is conditionally negative definite if and only if $g$ is positive definite,
        \item $f$ is conditionally negative definite if and only if for any $t\geq 0$ the function $(x,y)\mapsto \exp(-t f(x,y))$ is positive semi-definite.
    \end{enumerate}
\end{lem}

The idea is to show that if $d(.,.)$ defines a structure function of a Gaussian field then for a class of functions $F(d(.,.))$ will define the structure function of a Gaussian field, as well. We believe that the first instance of the following argument was in \cite[§5]{Gangolli1967}. There, it was used for $F:\R\to\R_{\geq 0}$ with $F(x):=x^\alpha$ for $0<\alpha\leq 1$. We present a version of \cite[Prop.\ 2.5]{Istas2012} (cf.\ also \cite[Lem.\ 3.1]{Istas2005}).

\begin{lem}\label{lem_BrownianFieldToFractionallevy}
    If there exists a L\'evy Brownian field on a metric space $(U,d)$ then the there exists a fractional L\'evy field for all $0<H< 1/2$ on the same index set.
\end{lem}
\begin{proof}
    Let $0<\alpha < 1$. Then one may directly verify using $e^{-u}\geq 1-u$ for all $u\in\R$ that
    \begin{align*}
        C_\alpha := \int_0^\infty \frac{1-e^{-u}}{u^{1+\alpha}} \dd u < \infty.
    \end{align*}
    Evidently, $C_\alpha > 0$. Thus, using the substitution $x\mapsto x/t$, it is immediate that
    \begin{align*}
        t^\alpha = \frac{-1}{C_\alpha} \int_0^\infty \frac{e^{-x t}-1}{x^{1+\alpha}} \dd x
    \end{align*}
    holds for any $t\geq 0$. 
    
    By Schoenberg's theorem (cf.\ first part of Lemma~\ref{lem_schoenbergsTheorem}) it suffices to check that $d(.,.)^{2H}$ is conditionally negative definite. So let $n\in\N$ and $\eta_1,\ldots,\eta_n\in U$ and $\lambda_1,\ldots,\lambda_n\in\R$ with $\lambda_1+\ldots + \lambda_n=0$. Then using the integral representation for $\alpha=2H$ we calculate that
    \begin{align*}
        \sum_{i,j=1}^n \lambda_i \lambda_j d(\eta_i,\eta_j)^{2H}
        & = \sum_{i,j=1}^n \lambda_i \lambda_j \frac{-1}{C_\alpha} \int_0^\infty \frac{e^{-x d(\eta_i,\eta_j)}-1}{x^{1+\alpha}} \dd x\\
        & = \frac{-1}{C_\alpha} \int_0^\infty \sum_{i,j=1}^n \lambda_i \lambda_j \frac{e^{-x d(\eta_i,\eta_j)}-1}{x^{1+\alpha}} \dd x.
    \end{align*}
    Using the existence of the L\'evy Brownian motion, the second part of Schoenberg's theorem (cf.\ Lemma~\ref{lem_schoenbergsTheorem} (b)), the fact that $\lambda_1+\ldots + \lambda_n=0$ and that $C_\alpha>0$ we obtain that the integrand is non-negative, which implies that the whole expression is non-positive. This is what we needed to show.
\end{proof}

Note that Lemma~\ref{lem_existenceLévyBrownianField2D} and Lemma~\ref{lem_BrownianFieldToFractionallevy} together imply Lemma~\ref{lem_mainExistence2dFBM}.

\begin{rem}
    The technique in Lemma~\ref{lem_BrownianFieldToFractionallevy} can be applied for a much broader class of functions. If $F$ is a Bernstein function, cf.\ \cite{SchillingSongVondracek2012}, and if $d(.,.)$ is the structure function of a L\'evy Brownian field then $F(d(.,.))$ is the structure function of a centred Gaussian field, as well, cf.\ \cite[Prop.\ 2.5]{Istas2012}.
\end{rem}

\begin{rem}
    It is unlikely that this approach generalises to higher dimensions, since the measure $\nu_G$ does not generalise to objects other than geodesics, except for hyperplanes in specific cases such as the models spaces $\mcl_\kappa$. The (local) existence of (fractional) L\'evy fields indexed by general Riemannian manifolds in higher dimensions is therefore an interesting open problem. 
\end{rem}

\end{document}